\documentclass[11pt,a4wide]{article}
\usepackage{setspace,a4wide}

\usepackage{amsmath}
\usepackage{amsfonts}
\usepackage{graphicx}
\usepackage[dvips]{color}

\usepackage[T1]{fontenc}
\newcounter{assump}
\newenvironment{proof}[1][Proof]{\textit{#1.} }{\ \rule{0.5em}{0.5em}}

\newtheorem{assum}[assump]{Assumption}

\makeatletter
\@addtoreset{equation}{section}

\makeatother

\newcounter{Fig}[figure]

\newcounter{Tab}[table]

  {\refstepcounter{Tab}
   \addtocounter{table}{1}
   \samepage\vspace{0.2cm}
   \centerline{#1} \nobreak
   \begin{verse}{Table \thesection.\arabic{Tab}:}
  }%
  {\end{verse} \vspace{0.2cm}}

\relax

\newcommand{\XX}{\mbox{\boldmath $X$}}
\newcommand{\YY}{\mbox{\boldmath $Y$}}
\newcommand{\ZZ}{\mbox{\boldmath $Z$}}

\newcommand{\ccc}{\mbox{\boldmath $c$}}

\newcommand{\xx}{\mbox{\boldmath $x$}}

\newcommand{\KK}{\mbox{\boldmath $K$}}

\newcommand{\hh}{\mbox{\boldmath $h$}}

\newcommand{\aaa}{\mbox{\boldmath $a$}}

\newcommand{\dd}{\mbox{\boldmath $d$}}
\newcommand{\bb}{\mbox{\boldmath $b$}}
\newcommand{\uu}{\mbox{\boldmath $u$}}
\newcommand{\vv}{\mbox{\boldmath $v$}}
\newcommand{\zb}{\mbox{\boldmath $z$}}
\newcommand{\zz}{\mbox{\boldmath $z$}}

\newcommand{\UU}{\mbox{\boldmath $U$}}

\def \Pb{{\mathbb P}}
\def \Rb{{\mathbb R}}

\def \Ac{{\mathcal A}}
\def \Rc{{\mathcal R}}
\def \Bc{{\mathcal B}}
\def \Gc{{\mathcal G}}
\def \Zc{{\mathcal Z}}

\def \Ec{{\mathcal E}}

\def \Fc{{\mathcal F}}
\def \Cc{{\mathcal C}}

\def \Hc{{\mathcal H}}
\def \Mc{{\mathcal M}}
\def \Nc{{\mathcal N}}
\def \Lc{{\mathcal L}}

\newcommand{\bqa}{\begin{eqnarray*}}
\newcommand{\eqa}{\end{eqnarray*}}
\newcommand{\bqan}{\begin{eqnarray}}
\newcommand{\eqan}{\end{eqnarray}}
\newcommand{\bqt}{\begin{quote}}
\newcommand{\eqt}{\end{quote}}
\newcommand{\bt}{\begin{tabbing}}
\newcommand{\et}{\end{tabbing}}
\newcommand{\bit}{\begin{itemize}}
\newcommand{\eit}{\end{itemize}}
\newcommand{\ben}{\begin{enumerate}}
\newcommand{\een}{\end{enumerate}}
\newcommand{\beq}{\begin{equation}}
\newcommand{\eeq}{\end{equation}}
\newcommand{\bdefi}{\begin{definition}}
\newcommand{\edefi}{\end{definition}}
\newcommand{\bpro}{\begin{proposition}}
\newcommand{\epro}{\end{proposition}}
\newcommand{\blem}{\begin{lemma}}
\newcommand{\elem}{\end{lemma}}
\newcommand{\bth}{\begin{theorem}}
\newcommand{\eth}{\end{theorem}}
\newcommand{\bco}{\begin{corollary}}
\newcommand{\eco}{\end{corollary}}
\newcommand{\bdes}{\begin{description}}
\newcommand{\edes}{\end{description}}
\newcommand{\bre}{\begin{remark}}
\newcommand{\ere}{\end{remark}}

\newtheorem{definition}{Definition}[section]
\newtheorem{proposition}[definition]{Proposition}
\newtheorem{lemma}[definition]{Lemma}
\newtheorem{theorem}[definition]{Theorem}
\newtheorem{corollary}[definition]{Corollary}
\newtheorem{remark}[definition]{\textit{ \textbf{Remark}}}

\newcommand{\ep}{\varepsilon}

\def\mds{\medskip}
\def\1{{\mathbf 1}}

\begin{document}

\title{{\Huge \bf Single-index copulae}}

\medskip

\author{{\sc JEAN-DAVID FERMANIAN} \\ {\it \small Crest-Ensae, 3 avenue Pierre Larousse, 92245 Malakoff cedex, France}
 \\ {\sc OLIVIER LOPEZ} \\ {\it \small Univ. Pierre et Marie Curie Paris VI, 4 place Jussieu, 75005 Paris, France.}}

\date{First version: March 2016. This version: April 2017}

\maketitle

\begin{abstract}
We introduce so-called ``single-index copulae''. They are semi-parametric conditional copulae whose parameter is an unknown ``link'' function of a univariate index only.
We provide estimates of this link function and of the finite dimensional unknown parameter.
The asymptotic properties of the latter estimates are stated. Thanks to some properties of conditional Kendall's tau, we illustrate our technical conditions with several usual copula families.
\medskip
\noindent

\textit{Key words}: Conditional copulae, kernel smoothing, single-index models

\medskip
\noindent
\textit{Running title}: Single-index copulae
\end{abstract}

\pagebreak

\section{Introduction}

\subsection{The framework of single-index dependence functions}

Since Sklar's theorem (1959), copula modeling has emerged as a very active field in theoretical and applied research.
Applications in finance, insurance, biology, medicine, hydrology, etc., are now countless. The origin of this success is the ability of splitting specification/inference/testing of a (complex) multivariate model into two separate (simpler) problems: the management of marginal distributions on one side, and the modelling of the dependence structure (copula) on the other side. See the books of Joe (1997) or Nelsen (1998) for a rigorous presentation of this field.

\mds

In practice, it is usual to introduce explanatory variables (also called ``covariates'') in a multivariate model, particularly in econometrics or financial risk management.
When we focus on the effect of these covariates on the underlying copulae, we need the concept of conditional copulae immediately (Patton, 2006). Conditional copulae are a natural way of linking conditional marginal distributions to get a multivariate conditional law and they have been applied extensively (see the surveys of Patton 2009, 2012). Recently, the rise of vine models (Aas et al., 2009) has extended the scope and the importance of conditional copulae.

\mds

Until now, most conditional copula models were parametric. For instance, they specify a given functional link between the copula parameters and an index $\beta'\zb$, $\zb$ being the underlying vector of covariates: see Rockinger and Jondeau (2006), Patton (2006), Rodriguez (2007), Batram (2007), among others. Alternatively, a fully nonparametric point of view has been proposed by Fermanian and Wegkamp (2012) or Gijbels et al. (2011). Such techniques rely on kernel smoothing, local polynomials or other tools in functional estimation. As a consequence, when the dimension of the vector of covariates of larger than three, such methods suffer from the well-know curse of dimension, and they become unfeasible in practice.

\mds

In this paper, we propose an intermediate solution, through a single-index  assumption on the underlying copula parameter. Therefore, only a finite-dimensional parameter and a univariate ``link'' function have to be estimated, avoiding the curse of dimension.
Note that Acar et al. (2011, 2013) and Abegaz et al. (2012) have proposed another alternative through local linear approximations of the link function between covariates and copula parameters. Nonetheless, the latter approach is based on a linearization (thus approximative) procedure and the number of unknown parameters is increasing very quickly with the dimension of $\zb$.
Moreover, Sabeti et al. (2014) have recently introduced additive models for copula, and have adapted the bayesian-type estimation procedure proposed by Craiu and Sabeti (2012).

\mds

To fix the ideas and the notations, let us consider an i.i.d. sample of observations $(\XX_i,\ZZ_i)$ in
$\Rb^d \times \Rb^p$, that are drawn from the law of $(\XX,\ZZ)$. The vector $\XX$ represents the endogenous vector,
and $\ZZ$ is the vector of covariates. We are interested in the
evaluation of the law of $\XX$ conditional on $\ZZ=\zb$, for arbitrary vectors $\zb$. This
conditional cdf is denoted by $F(\cdot | \zb)$. The (marginal)
law of $X_k$, $k=1,\ldots,d$, given $\ZZ=\zb$, is
denoted by $F_k(\cdot | \zb)$. We introduce the unobserved random
vector $\UU_{\zb}=(U_{1,\zb},\ldots,U_{d,\zb})$, where $U_{k,\zb}=F_k(X_k |\zb)$, $k=1,\ldots,d$.
To simplify notations and when there is no ambiguity, $\UU_{\zb}$ will often be denoted by $\UU$.
By definition, the law of $\UU_{\zb}$ knowing $\ZZ=\zb$ is the conditional
copula of $\XX$ knowing $\ZZ=\zb$, denoted by $C(\cdot | \zb)$.

\mds

First, we live in a parametric framework. A natural model specification would be to assume that, for any
$\uu \in [0,1]^d$ and any $\zb\in \Rb^p$,
$$ C(\uu | \zb ) = C_{\theta(\zb)} (\uu),$$
where $\theta:\Rb^p \rightarrow \Rb^q$ maps the vector of covariates
to  the (true) parameter of the conditional copula knowing
$\ZZ=\zb$, and $\mathcal{C}=\{C_{\theta}:\theta\in \Theta \subset
\Rb^q\}$ denotes a parametric family of copulae. The copula density
of $C_{\theta}$ is supposed to exist and is denoted by $c_\theta$. To simplify, this density
is assumed to be continuous for every $\theta\in\Theta$, and $\Theta$ is a compact subset.

\mds

Second, since the single-index assumption is related to the dependence
function among the components of $\XX$, given the
covariates, this means there exists an unknown
function $\psi$ s.t.
\begin{equation}
\theta(\zb)=\psi (\beta_0,\beta_0'\zb),
\label{single_index_cop}
\end{equation}
where the true parameter $\beta_0\in \Bc$, a compact subset in $\mathbb{R}^m$.
To identify the parameter $\beta_0$, let us assume that the first component of $\beta_0$, that is $\beta_{0,1}$, is equal to one~\footnote{and the estimate of $\beta_{0,1}$
is always one, obviously}.
Note that it is necessary to indicate the dependency of $\theta(\zb)$ on $\beta_0$ explicitly.
Indeed, given $\ZZ=\zb$, we need to know the true parameter value to be able to calculate the index $\theta'_0\zb$ and then to evaluate the underlying conditional copula~\footnote{This is standard in the single-index literature: see Ichimura (1993), for instance, where $E[Y|\ZZ=\zb,\beta_0]$ is a function a real number $h(\beta_0,\zb)$ (with obvious notations).}. Then, under the single-index assumption~(\ref{single_index_cop}), $C(\cdot | \zb)$ does depend on $(\beta,\beta'\zb)$ if the underlying parameter is
$\beta$. Therefore, this function will be denoted equivalently $C_{\beta}(\cdot | \beta'\zb)$ too.

\mds

{\bf\it Example 1:} assume that the conditional copula of $\XX $ given $\ZZ=\zb$ is a $d$-dimensional Gaussian copula $C_\Sigma$. Its associated correlation matrix is
$ \Sigma = \left[  \1(i=j)+ \1(i\neq j) h(\beta_0'\zb)  \right]_{1\leq i,j \leq d},$
 for some unknown function $h:\Rb \longrightarrow (-1/(d-1),1)$.

\mds

We stress that Assumption~(\ref{single_index_cop}) does not mean that $C(\cdot | \zb)$, the conditional copula of $\XX$ knowing $\ZZ=\zb$, is equal to the
conditional copula of $\XX$ knowing $\beta_0'\ZZ=\beta_0'\zb$ (denoted by $\tilde C(\cdot | \beta_0'\zb)$). Indeed, in the former case, the relevant margins are the cdfs' $F_k(\cdot |\zb)$,
$k=1,\ldots,d$, and in the latter case, we need to consider the cdfs' $\tilde F_k(\cdot | \beta_0'\zb): x_k\mapsto P(X_k \leq x_k | \beta_0'\zb)$.
To avoid any confusion, let us denote $\tilde{\UU}_\beta = (\tilde{F}_1(X_1 | \beta'\ZZ),\ldots, \tilde{F}_d(X_d | \beta'\ZZ))$, and $\tilde C(\cdot |\beta'\ZZ=y)$ is the copula of
$\tilde{\UU}_{\beta}$ knowing $\beta'\ZZ=y$. The conditional copulae $C(\cdot | \zb)$ and $\tilde{C}(\cdot | \beta_0'\zb)$ are identical only when
$\ZZ$ provides the same information as $\beta_0'\ZZ$ to explain every margin $X_k$, i.e. when $F_k(\cdot |\zb)=\tilde F_k(\cdot | \beta_0'\zb)$ a.e. for every $k$: see Fermanian and Wegkamp (2012) for a discussion.

\mds

Our single-index copula assumption~(\ref{single_index_cop}) is relevant theoretically and empirically. Indeed, in general, covariates have a different influence on the conditional margins and on the underlying conditional copula. See Acar et al. (2012) for a discussion, and empirical illustrations.
Mainly for practical reasons, some authors assume even more that the conditional copula of $\XX$ given $\ZZ=\zb$ does not depend on $\zb$.
This is the so-called ``simplifying assumption'' (Derumigny and Fermanian, 2017, e.g.), commonly used in vine models.

\mds

\subsection{The M-estimate criterion}

Single-index models are well-known in the world of semiparametric statistics. The theory of M-estimators has started with the seminal papers of Klein and Spady (1993) in the case of the so-called binary response model, and Ichimura (1993) for the general single-index regression model. Sherman (1994), Delecroix and Hristache (1999) extended this approach. H\"ardle et al. (1993) and Delecroix et al. (1999) discussed the choice of the bandwidth for the nonparametric estimation of the link function.
Alternatively, the so-called average derivative method has been developed in parallel by Stoker (1986), Powell, Stock and Stoker (1989), H\"ardle and Stoker (1989), among others.

\mds

In this paper, we rely on M-estimators of single-index models, but related to the parameter of the underlying copula only. If we
were able to observe a sample of the random vector $\UU$, i.e. $\UU_i$, $i=1,\ldots,n$,
then our ``naive'' estimator of $\beta_0$ could be
$$\hat{\beta}_{naive}=\arg \max_{\beta\in \mathcal{B}}\sum_{i=1}^n \ln
c_{\hat\psi (\beta,\beta'\zb_i)}(\UU_i),$$ for some function $\hat\psi$ that estimates $\psi(\cdot,\cdot)$ consistently.

\mds

Since we do not observe realizations of $\UU$, we have to replace
the unknown vectors $\UU_i$ by some estimates $\hat \UU_i$,
given $\ZZ_i$, providing a so-called pseudo-sample
$\hat\UU_1,\ldots,\hat\UU_n$. Then, a natural idea is to define our estimator by
\beq
\hat{\beta}=\arg \max_{\beta\in \mathcal{B}}\sum_{i=1}^n \hat\omega_{i,n}\ln c_{\hat\psi (\beta,\beta'\ZZ_i)}(\hat\UU_{i}),
\label{true_criterion}
\eeq
for some sequence of trimming functions $\hat\omega_{i,n}$. Typically, they are of the type
$\hat\omega_{i,n}=\1(\hat\UU_i \in\mathcal{E}_n, \ZZ_i \in \Zc)$, for some non decreasing sequence of subsets $\mathcal{E}_n$ in $[0,1]^d$, and some $\Zc \subset \Rb^p$.
Such trimming functions allow to control some boundary effects and the uniform convergence of our kernel estimates.
For technical reasons, we choose strictly increasing trimmings on the $\UU$-side, i.e. $\cup_n\mathcal{E}_n=(0,1)^d$. This choice makes it necessary to control explicitly the behavior of $\UU$ close to the boundary of $[0,1]^d$.
This pretty delicate task requires several regularity assumptions but the problem has already been met in the literature (see Tsukahara 2005, for instance).
Moreover, we set a fixed trimming for $\Zc$ (i.e. $\Zc \subset \Rb^p$ strictly). This will not create any bias, because the law of the $\UU$ knowing $\ZZ\in \Zc$, is just $c_{\psi(\beta_0,\beta_0'\zb)}(\uu ) \1(\zb\in \Zc)/\Pb (\ZZ\in \Zc)$. Thus, this law depends on the true parameter $\beta_0$. See Assumption~\ref{A_trim}.

\mds

\bre
Actually, fixed trimming functions for $\hat\UU_i$ could be chosen instead, i.e. $\Ec_n=\Ec \subset [a,1-a]^d$ for some $a>0$ and every $n$. They would induce consistent estimates without having to impose regularity conditions on the copula density close to the frontier of $[0,1]^d$. But the asymptotic behavior of $\hat \beta$ would be more complex. Typically, it would be asymptotically normal, but after removing
 an annoying bias that cannot be evaluated easily. Moreover, beside a small loss of efficiency, this would forbid to model the tail dependence behaviors, a feature that is important in a lot of fields. That is why we have chosen $\hat\beta$, as defined by~(\ref{true_criterion}).
\ere

\mds

\section{Consistency}
\label{consistency}

\subsection{The convergence of single-index estimators}

\begin{assum}
\label{A_trim}
Let us set $\Zc:=[-M_z,M_z]^p$ and $\Ec_n=[\nu_n,1-\nu_n]^d$ for some positive sequence $(\nu_n)$, $\nu_n \in(0,1/2)$, $\nu_n\rightarrow 0$.
The trimming functions are
$\omega_n:[0,1]^d \times \Rb^p \rightarrow [0,1]$, $(\uu,\zb)\mapsto \1(\uu\in \Ec_n,\zb\in \Zc)$.
\end{assum}

We set $\hat\omega_{i,n} = \omega_n(\hat\UU_i,\ZZ_i)$ simply.
For the sake of completeness, we introduce $\omega_{i,n}:=\omega_n(\UU_i,\ZZ_i)$, the trimming function when $\UU_i$ is known, and $\omega_i=\omega_{i,\infty} = \1(\ZZ_i \in \Zc)$.
Typically, the constant $M_z$ is chosen so that the density of $\ZZ$ wrt the Lebesgue measure exists and is lower bounded, i.e. $\inf_{\zz\in \Zc} f_{\ZZ}(\zz)\geq f_0>0$. This will be assumed hereafter, even if this is not mandatory at this stage.

\begin{assum}
\label{A_identif}
The parameter $\beta_0$ is identifiable, i.e. two different parameters induce two different laws of $\UU_{\ZZ}$, knowing $\ZZ\in\Zc$.
The function $M:\Bc \rightarrow \Rb$, $\beta \mapsto E[\ln c_{\psi(\beta,\beta'\ZZ)}(\UU_{\ZZ})\, | \, \ZZ \in \Zc] $ is continuous and uniquely maximized at $\beta=\beta_0$.
There exists a measurable function $h$ and $\alpha >1$ s.t., for every $\zb\in \Zc$,
 \beq
\sup_{\beta\in\Bc} |\ln c_{\psi(\beta,\beta'\zb)}(\UU_{\zb})| \leq
h(\UU_{\zb},\zb),\;\text{with}\,\; E[h^{\alpha}(\UU_{\ZZ},\ZZ).\1(\ZZ\in\Zc)]<\infty.
\label{A_UnifLLN}
\eeq
\end{assum}

The latter assumption is usual for maximum likelihood estimation purpose.
The limiting objective function is here
$$ M(\beta) := E\left[\ln c_{\psi(\beta,\beta'\zb_i)}(\UU_{i}) \,| \, \ZZ_i\in \Zc \right],$$
that is maximized at $\beta=\beta_0$. Obviously, all expectations $E[\cdot]$ have to be understood wrt the true law of $(\XX,\ZZ)$ that depends on $\beta_0$ only.
Note that, due to our trimming functions, we are dealing with a M-estimator of $\beta$ instead of a usual MLE formally, at the cost of a (small) loss of efficiency.

\begin{assum}
\label{A_unifapprox}
\beq
\sup_{\zb\in\Zc}\sup_{\beta \in \Bc}\left|\hat{\psi}(\beta,\beta'\zb)-\psi(\beta,\beta'\zb)\right| = o_P(1).
\label{unif_psi}
\eeq
Moreover, the pseudo-observations $\hat U_{i,k}$ belong to $(0,1)$, $k=1,\ldots,d$, $i=1,\ldots,n$, and there exists a deterministic sequence $(\delta_n)$, $\delta_n = o(\nu_n)$, s.t.
\beq
\lim\sup_{n\rightarrow \infty}\sup_{i} | \hat\UU_i - \UU_i |.\1(\ZZ_i\in \Zc)/\delta_n =1 \;\; \text{a.e.}
\label{unif_uu}
\eeq
\end{assum}

In particular,~(\ref{unif_uu}) implies that $\sup_{i} | \hat\UU_i - \UU_i |.\1(\ZZ_i\in \Zc)=O_P(\delta_n) $.
These assumptions have to be checked for any particular single-index model (see the assumptions (A1) and (A2) in Subsection~\ref{psi_mgt})
and for any particular estimate of the marginal cdfs'.

\mds

Now, we recall the definition of reproducing u-shaped functions, as introduced in Tsukahara (2005).
\begin{definition}

\begin{itemize}
\item A function $f:(0,1)\rightarrow (0,\infty)$ is called u-shaped if it is symmetric about $1/2$ and decreasing on $(0,1/2]$.
\item For $\beta\in (0,1)$ and a u-shaped function $r$, define
$$
r_\beta(u)= \left\{
\begin{array}{lcc}
r(\beta u) & \text{if} & 0<u\leq 1/2; \\
r(1-\beta(1-u)) & \text{if} &  1/2 < u \leq 1.
\end{array}
\right.$$
If, for every $\beta > 0$ in a neighborhood of $0$, there exists a constant $M_\beta$ such that $r_\beta < M_\beta\, r$
on $(0, 1)$, then $r$ is called a reproducing u-shaped function.
\item
We denote by $\Rc$ the set of univariate
reproducing u-shaped functions. The set $\Rc_d$ is the set of functions $r:(0,1)^d\rightarrow \Rb^+$, $r(\uu)=\prod_{k=1}^d r_k(u_k)$, and $r_k \in \Rc$ for every $k$. Moreover, $r_\beta(\uu)=\prod_{k=1}^d r_{k,\beta}(u_k)$.
\end{itemize}
\end{definition}

Typically, the usual functions in $\Rc$ are of the type $r(u)= C_r u^{-a}(1-u)^{-a},$ for some positive constants $a$ and $C_r$.

\begin{assum} There exist some functions $r$, $\tilde r_1,\ldots,\tilde r_d$ in $\Rc_d$ s.t., for every $\uu\in (0,1)^d$,
$$\sup_{\theta\in \Theta}\left| \nabla_{\theta}\ln c_{\theta}(\uu) \right| \leq r(\uu),\;E\left[ r(\UU_{\ZZ}) \1(\ZZ\in \Zc)\right]<\infty,$$
$$ \sup_{\theta\in \Theta}\left| \partial_{u_k}\ln c_{\theta}(\uu) \right| \leq \tilde r_k(\uu),\;\text{for every } k=1,\ldots,d,\; \text{and}$$
$$ \sup_{k=1,\ldots,d}E\left[U_k(1-U_k) \tilde r_k(\UU_{\ZZ})\1(\ZZ\in \Zc) \right] <\infty.$$
\label{A_moments}
\end{assum}
The latter conditions of moments are easily satisfied for most copula models. They are close to those of Assumption (A.1) in Tsukahara (2005).

\begin{theorem}
\label{Th_consistency}
Under the assumptions~\ref{A_trim}-\ref{A_moments}, the estimator $\hat\beta$ given by~(\ref{true_criterion}) tends to $\beta_0$ in probability, when $n$ tends to infinity.
\end{theorem}

\begin{proof}
For inference purpose and a given sample, the sample size that we use is actually $\hat n_i=\sum_{i=1}^n \hat\omega_{i,n}$. This random number is close to $n_i=\sum_{i=1}^n \omega_{i,n}$, the sample size if the $\UU_i$ were observable.
Let us introduce
\begin{eqnarray*}
M_n(\beta) &:=&
\frac{1}{n_i+1}\sum_{i=1}^n \hat\omega_{i,n}\ln c_{\hat\psi (\beta,\beta'\zb_i)}(\hat\UU_{i}),\\
M^*_n(\beta) &:=& \frac{1}{n_i+1}\sum_{i=1}^n  \hat\omega_{i,n}\ln c_{\psi(\beta,\beta'\zb_i)}(\UU_{i}),\\
M^{**}_n(\beta) &:=& \frac{1}{n_i+1}\sum_{i=1}^n  \omega_{i}\ln c_{\psi(\beta,\beta'\zb_i)}(\UU_{i}).
\end{eqnarray*}
Note that $\hat \beta$ is the optimizer of $M_n(\cdot)$ because neither $n_i$ or $\hat n_i$ is a function of the underlying parameter $\beta$.
By assumption, $\beta_0$ maximizes $M(\beta)$ over $\Bc$. To prove the consistency of $\hat\beta$, it is sufficient to show
that $\sup_{\beta\in \mathcal{B}}|M_n(\beta)-M(\beta)|=o_P(1)$.

\mds

We first show that $\sup_{\beta\in
\mathcal{B}}|M_n(\beta)-M_n^*(\beta)|=o_P(1).$
Simple calculations provide
\bqa
| M_n(\beta) - M^*_n(\beta) | & \leq & \frac{1}{ n_i+1}\sum_{i=1}^n \hat\omega_{i,n} \sup_{\theta\in \Theta}\left|
 \frac{\nabla_{\theta} c_{\theta}(\hat\UU_i)}{c_{\theta}(\hat\UU_i)}\right|
 \cdot\left|\hat{\psi}(\beta,\beta'\ZZ_i)-\psi(\beta,\beta'\ZZ_i)\right| \\
&+& \frac{1}{n_i+1}\sum_{i=1}^n \left| \hat\omega_{i,n}\frac{\nabla_{\uu} c_{\psi
(\beta,\beta'\zb_i)}}{c_{\psi (\beta,\beta'\zb_i)}}(\UU_i^*)\cdot
(\hat\UU_{i}- \UU_i) \right| :=T_1(\beta) + T_2(\beta), \eqa
for some vectors $\UU_i^*$ s.t. $|\UU_i - \UU_i^*|\leq |\UU_i - \hat\UU_i|$ for all $i$.

Let us deal with $T_1(\beta)$.
By~(\ref{unif_psi}), $\sup_\beta\sup_i\left|\hat{\psi}(\beta,\beta'\ZZ_i)-\psi(\beta,\beta'\ZZ_i)\right| = o_P(1)$. Then, it suffices to prove that
$$ \frac{1}{ n_i+1}\sum_{i=1}^n \hat\omega_{i,n} \sup_{\theta\in \Theta}\left|
 \frac{\nabla_{\theta} c_{\theta}(\hat\UU_i)}{c_{\theta}(\hat\UU_i)}\right|=O_P(1).$$
For every $\ep>0$ and $A>0$, we have
\bqa
\lefteqn{ \Pb\left( \frac{1}{ n}\sum_{i=1}^n \hat\omega_{i,n} \sup_{\theta\in \Theta}\left|
 \frac{\nabla_{\theta} c_{\theta}(\hat\UU_i)}{c_{\theta}(\hat\UU_i)}\right|  > A \right) \leq
 \Pb\left( \sup_i|\hat \UU_i - \UU_i| > 2\delta_n,\ZZ_i\in \Zc \right)   }\\
&+ & \Pb\left( \frac{1}{ n}\sum_{i=1}^n \hat\omega_{i,n} \1(|\hat \UU_i - \UU_i| \leq 2\delta_n).\sup_{\theta\in \Theta}\left|
 \frac{\nabla_{\theta} c_{\theta}(\hat\UU_i)}{c_{\theta}(\hat\UU_i)}\right|  > A \right) .\eqa
By~(\ref{unif_uu}), the first term on the r.h.s. above is less than $\ep$ when $n$ is large.
To manage the last term on the r.h.s., consider an arbitrary index $i$ s.t. $|\hat\UU_i - \UU_i| \leq 2\delta_n$ and $\hat\omega_{in}=1$.
Since $\delta_n=o(\nu_n)$, we can assume that, for every $k=1,\ldots,d$, we have
$$   U_{i,k}/2 \leq \hat U_{i,k}  \; \text{ if} \;\; \hat U_{i,k}\leq 1/2,\;\text{and} $$
$$   (1-U_{i,k})/2 \leq (1-\hat U_{i,k})  \;\; \text{ if} \; \hat U_{i,k}> 1/2. $$
For the $k$-th of the u-shaped functions $r_k$ that define $r$, we deduce
$$  r_k(\hat U_{i,k}) \leq r_k(U_{i,k}/2) \;\; \text{ if} \; \hat U_{i,k}\leq 1/2,\;\text{and} $$
$$  r_k(\hat U_{i,k}) \leq r_k(1 - (1- U_{i,k})/2) \;\; \text{ if} \; \hat U_{i,k}> 1/2. $$
In other words, $ r_{k}(\hat U_{i,k}) \leq r_{k,1/2}(U_{i,k})$ for such $i$ and every $k$.
Then, Assumption~\ref{A_moments} implies
\bqa
\lefteqn{ \frac{1}{n}\sum_{i=1}^n \sup_{\theta\in \Theta}\left|
 \frac{\nabla_{\theta}c_{\theta}(\hat\UU_i)}{c_{\theta}(\hat\UU_i)}\right|\hat\omega_{i,n} \1(|\hat\UU_i - \UU_i|\leq 2\delta_n)  \leq
\frac{1}{n}\sum_{i=1}^n r(\hat\UU_i) \hat\omega_{i,n} \1(|\hat\UU_i - \UU_i|\leq 2\delta_n) }\\
& \leq & \frac{1}{n}\sum_{i=1}^n r_{1/2}(\UU_i)\omega_i \leq \frac{M^d_{1/2}}{n}\sum_{i=1}^n r(\UU_i)\omega_i,   \hspace{8cm}
\eqa
which is integrable. We get
$$\Pb\left( \frac{1}{ n}\sum_{i=1}^n \hat\omega_{i,n} \sup_{\theta\in \Theta}\left|
 \frac{\nabla_{\theta} c_{\theta}(\hat\UU_i)}{c_{\theta}(\hat\UU_i)}\right|  > A \right)
\leq  \ep + \frac{ M^d_{1/2} E[r(\UU_i)\omega_i]  }{A}  < 2\ep,$$
when $A$ and $n$ are sufficiently large.
Since $n_i/n$ tends to a positive constant a.e., we deduce $\sup_\beta T_1(\beta)=o_P(1)$.

\mds

By a slightly more subtle reasoning, we can obtain $\sup_\beta T_2=o_P(1)$.
For every $\ep>0$,
\bqan
\lefteqn{ \Pb(T_2(\beta) > \ep) \leq  \Pb(\sup_i |  \hat\UU_{i}- \UU_i | > 2\delta_n,\ZZ_i\in\Zc) \label{ineg_T2beta} }\\
&+&  \Pb\left( \frac{1}{n_i+1}\sum_{i=1}^n  \hat\omega_{i,n}  \1(|\hat\UU_{i}- \UU_i|\leq 2\delta_n) \sup_\beta\left| \frac{\nabla_{\uu} c_{\psi
(\beta,\beta'\zb_i)}}{c_{\psi (\beta,\beta'\zb_i)}}(\UU_i^*).(\hat\UU_i - \UU_i) \right|   > \ep \right),
\nonumber
\eqan
and the first term on the r.h.s. is smaller than $\ep$ when $n$ is large.
Due to Assumption~\ref{A_moments} and for every $\varepsilon>0$, there exists $\eta \in (0,1/2)$ s.t.
$$\sup_{k=1,\ldots,d} E[\tilde r_{k,1/2}(\UU_{\ZZ}) \{U_k\1(U_k < \eta) + (1-U_k)\1(U_k > 1- \eta)\}.\1(\ZZ\in\Zc)] < \varepsilon^2.$$
By applying the Markov's inequality, we deduce
\bqa
\lefteqn{
\Pb\left( \frac{1}{n}\sum_{i=1}^n  \hat\omega_{i,n}  \1(|\hat\UU_{i}- \UU_i|\leq 2\delta_n) \sup_\beta\left| \frac{\nabla_{\uu} c_{\psi
(\beta,\beta'\zb_i)}}{c_{\psi (\beta,\beta'\zb_i)}}(\UU_i^*).(\hat\UU_i - \UU_i) \right|   > \ep \right)  }\\
&\leq & \Pb\left(
 \frac{1}{n}\sum_{k=1}^d \sum_{i=1}^n \hat\omega_{i,n} \1(|\hat\UU_i - \UU_i| \leq 2\delta_n)
\left|  \tilde{r}_k( \UU_i^*) \right| \cdot | \hat U_{i,k}- U_{i,k} | >\ep \right)  \\
& \leq & \Pb\left( \frac{1}{n}\sum_{k=1}^d \sum_{i=1}^n \omega_{i} \1(|\hat\UU_i - \UU_i| \leq 2\delta_n)
 \tilde{r}_{k,1/2}( \UU_i) \right. \\
&\cdot & \left.
\left\{ | \hat U_{i,k}- U_{i,k} |\1\{  \eta \leq U_{i,k} \leq 1- \eta\}
+ U_{i,k}\1(U_{i,k} < \eta) + (1-U_{i,k})\1(U_{i,k} > 1- \eta)    \right\} >\ep \right)\\
&\leq & \Pb\left(
\frac{2\delta_n}{n}\sum_{k=1}^d \sum_{i=1}^n \omega_{i}
 \tilde{r}_{k,1/2}( \UU_i) \cdot \1\{  \eta \leq U_{i,k} \leq 1- \eta\}>\ep/2 \right)
+ \frac{2\varepsilon^2}{\ep} \\
&\leq & \Pb\left(
\frac{2\delta_n}{n\eta}\sum_{k=1}^d \sum_{i=1}^n \omega_i U_{i,k}(1-U_{i,k})
  \tilde{r}_{k,1/2}( \UU_i)  > \ep/2\right)
+ 2\ep ,
\eqa
and then
$$ \Pb(T_2(\beta) > \ep) \leq  3\ep + \frac{4d\delta_n \sup_k E[\omega_{i} U_{i,k}(1-U_{i,k})
  \tilde{r}_{k,1/2}( \UU_i) ]}{\eta\ep},$$
that is less than $4\ep$ when $n$ is sufficiently large, because of Assumption~\ref{A_moments}.
Note that we have used $\hat U_{i,k}\in (0,1)$ for every $i=1,\ldots,n$ and $k=1,\ldots,d$ above.
Since $\varepsilon$ may be arbitrarily small and $n_i/n$ tends to a constant a.e., we get $\sup_\beta T_2(\beta)=o_P(1)$, and we have proved that $\sup_{\beta\in \mathcal{B}}|M_n(\beta)-M_n^*(\beta)|=o_P(1)$.

\mds

Second, due to Assumption~\ref{A_identif} and for every $\ep>0$, we have
\bqan
\lefteqn{\Pb \left( \sup_\beta|\frac{1}{n}\sum_{i=1}^n  (\hat\omega_{i,n}-\omega_{i,n})\ln c_{\psi(\beta,\beta'\zb_i)}(\UU_{i}) | > \ep \right) \nonumber }\\
& \leq  &
\Pb \left( \frac{1}{n}\sum_{i=1}^n  \1(\UU_i\in\mathcal{E}_n,\hat\UU_i\not\in\mathcal{E}_n, \ZZ_i\in \Zc ) \sup_\beta|\ln c_{\psi(\beta,\beta'\zb_i)}(\UU_{i}) | > \ep /2 \right)  \nonumber\\
& +  &
\Pb \left( \frac{1}{n}\sum_{i=1}^n  \1(\UU_i\not\in\mathcal{E}_n,\hat\UU_i \in\mathcal{E}_n,\ZZ_i\in \Zc) \sup_\beta|\ln c_{\psi(\beta,\beta'\zb_i)}(\UU_{i}) | > \ep /2 \right) \nonumber\\
& \leq & \frac{2}{\ep} E\left[ \left\{ \1(\UU_i\in\mathcal{E}_n,\hat\UU_i\not\in\mathcal{E}_n )
+ \1(\UU_i\not\in\mathcal{E}_n,\hat\UU_i\in\mathcal{E}_n ) \right\}\cdot \1(\ZZ_i\in \Zc)
h(\UU_{i},\ZZ_i)     \right]\nonumber\\
& \leq &
\frac{4}{\ep} E\left[  \1( |\hat\UU_i - \UU_i|>2\delta_n,\ZZ_i\in \Zc)
h(\UU_{i},\ZZ_i) \right] \label{consis_omeg} \\
&+&
\frac{2}{\ep} E\left[ \left\{ \1(\UU_i\in\mathcal{E}_n,\hat\UU_i\not\in\mathcal{E}_n )
+ \1(\UU_i\not\in\mathcal{E}_n,\hat\UU_i\in\mathcal{E}_n ) \right\} \1(|\hat\UU_i - \UU_i|\leq 2\delta_n, \ZZ_i\in \Zc)
h(\UU_{i},\ZZ_i) \right]. \nonumber
\eqan
But we have for any $i$
\bqa
\lefteqn{ \1(|\hat\UU_i - \UU_i|\leq 2 \delta_n)\cdot \{\1(\UU_i \not\in \mathcal{E}_n,\hat\UU_i \in \mathcal{E}_n) + \1(\UU_i \in \mathcal{E}_n,\hat\UU_i \not\in \mathcal{E}_n) \} }\\
& \leq & 2\sum_{k=1}^d \left\{ \1(U_{i,k} \in [\nu_n-2\delta_n,\nu_n+2\delta_n])
+ \1(1-U_{i,k} \in [\nu_n-2\delta_n,\nu_n+2\delta_n]) \right\},
\eqa
that tends to zero a.e. when $n$ tends to infinity. Invoking the dominated convergence Theorem and~(\ref{A_UnifLLN}), this proves that the second term of the r.h.s. in~(\ref{consis_omeg}) is less than $\ep$ when $n$ is large enough.

Moreover, due to Assumption~\ref{A_identif} and H\"older's inequality,
\bqa
\lefteqn{  E\left[  \1( |\hat\UU_i - \UU_i|>2\delta_n,\ZZ_i\in \Zc)
h(\UU_{i},\ZZ_i)  \right] }\\
& \leq &
E\left[h(\UU,\ZZ)^\alpha .\1(\ZZ\in \Zc)\right]^{1/\alpha}
\Pb\left( |\hat\UU_i - \UU_i|>2\delta_n, \ZZ_i\in \Zc \right)^{1-1/\alpha},
\eqa
that is less than $\ep$ when $n$ is large enough (Assumption~\ref{unif_uu}).
This provides
$$ \sup_\beta\left|\frac{1}{n}\sum_{i=1}^n  (\hat\omega_{i,n}-\omega_{i,n})\ln c_{\psi(\beta,\beta'\zb_i)}(\UU_{i}) \right|  = o_P(1).$$

\mds

Similarly, we prove
$ n^{-1}\sup_\beta\left| \sum_{i=1}^n  (\omega_{i,n}-\omega_{i})\ln c_{\psi(\beta,\beta'\zb_i)}(\UU_{i}) \right| = o_P(1).$
We deduce easily
$\sup_{\beta\in
\mathcal{B}}|M_n^*(\beta)-M_n^{**}(\beta)|=o_P(1)$ because $n_i/n$ tends to a constant a.e.

\mds

To conclude the proof, we can apply a usual uniform law of large
numbers. For instance, Lemma 2.4 in Newey and McFadden (1994) tells us that~(\ref{A_UnifLLN}) insures that $\sup_{\beta\in \mathcal{B}}|M_n^{**}(\beta)-M(\beta)|=o_P(1)$. Therefore, we get that $\hat\beta$ tends to $\beta_0$ in probability.
\end{proof}

\mds

Until now, we have not specified how we estimate the link function $\psi$ and the pseudo-observations $\hat\UU_i$. This will be the subject of the next two subsections.

\subsection{Estimation of the link function $\psi$}
\label{psi_mgt}

For inference purpose, we need a relationship between the previous link function $\psi$ and some quantities that can be estimated empirically.
Typically, there are two possibilities in practice.
\begin{itemize}
\item[(A1)] There exists a known ``explicit'' functional $\Psi$ s.t., for any $\beta\in\Rb^m$,
\beq \psi (\beta,\beta'\zb) = \Psi\left(C_{\beta}(\cdot | \beta'\zb)    \right).
\label{hypPsiC}
\eeq
\item[(A2)] There exists a known ``explicit'' functional $\Psi$ s.t., for any $\beta\in\Rb^m$,
\beq
 \psi (\beta,\beta'\zb) = \Psi\left(H_{\beta}(\cdot |  \beta'\zb)    \right),
 \label{hypPsiH}
 \eeq
where $H_{\beta}(\cdot | y)$ denotes the cdf of $(\XX,\ZZ)$ conditional on $\beta'\ZZ=y$ and given $\beta$.
\end{itemize}
Note that (A1) is a particular case of (A2), due to Sklar's theorem. Both situations are distinguished for pedagogical reasons (see the discussion on Kendall's tau below).

\mds

In numerous practical situations, Assumptions~(\ref{hypPsiC}) and~(\ref{hypPsiH}) are simply moment-like conditions, as in the GMM methodology:
there is a map $g:\Rb^{\bar m} \rightarrow \Rb^q$, $\bar m \geq m$, such that
$$ \theta(\zb) = g(m_1(\beta_0,\beta'_0 \zb),\ldots,m_{\bar m}(\beta_0,\beta'_0 \zb)),$$
where $m_k(\beta,y)\in \Rb$, $k=1,2,\ldots$, are ``moment'' relations based on the underlying distributions.
In the case of~(\ref{hypPsiC}), these moment relations are directly linked to conditional copulae by
\bqan
\lefteqn{ m_k(\beta,y)=E[\chi_k(\UU_{\ZZ},\beta'\ZZ) | \beta'\ZZ=y] = E[E[\chi_k(\UU_{\ZZ},\beta'\ZZ) | \ZZ] |\beta'\ZZ=y] }\nonumber\\
&=&E[ \int \chi_k(\uu,\beta'\ZZ)\, C(d\uu |\ZZ) |\beta'\ZZ=y ] =\int \chi_k(\uu,y)\, C_{\beta}(d\uu |\beta'\ZZ= y),
\label{formMoment2}
\eqan
for some known functions $\chi_k$, $k=1,\ldots,\bar m$.

\mds

In the case of~(\ref{hypPsiH}), there exist some ``moments'' $m_k(\beta,y)\in \Rb$, $k=1,2,\ldots$, based on the underlying distribution of $(\XX,\ZZ)$ given $\beta'\ZZ=y$:
\beq
m_k(\beta,y)=E[\chi_k(\XX,\ZZ) | \beta'\ZZ=y]=\int \chi_k(\xx,\zb)\, H_\beta(d\xx,d\zb |\beta'\ZZ=y ).
\label{formMoment}
\eeq

During the estimation procedure, the latter moments $m_k$, or more generally the cdfs' $C_\beta(\cdot | \beta'\zb)$ and
$H_\beta (\cdot |\beta'\zb)$ in (A1) and (A2), will be replaced by some empirical counterparts.
The formalism of (A2) behaves nicer than (A1), because it is simpler to work with the observations $(\XX_i,\ZZ_i)$ directly rather than with vectors $\UU_{i}$ (i.e. some i.i.d. realizations of the random vector $\UU_{\ZZ}$). Indeed, since $\UU_{\ZZ}$ cannot be observed, the latter quantities $\UU_i$ have to be estimated too, adding another level of complexity.

\mds

Note that, in theory, we are always living under (A1) or (A2), but the mappings $\Psi$ may be not explicit. In practice, this means there may exist some complex nonlinear relationships between the underlying distributions and the parameters, given $\ZZ$. It is similar to argue that the GMM approach is potentially ``universal'', but it is not (easily) practicable without any moments that can be written analytically.

\mds

{\bf\it Example 2: Spearman's rho.}

A natural candidate is given by
$m_k(\beta,\beta' \zb)=\rho(\beta,\beta' \zb)$, a multivariate extension of the usual Spearman's rho, defined by
$$ \rho(\beta,y) = \int \left(C_{\beta}(\uu | \beta'\ZZ=y) - \prod_{j=1}^d u_j\right) \, d\uu.$$
Through a $d$-dimensional integration by parts, check that this moment is of the type~(\ref{formMoment2}). Therefore, we work under (A1).
Other definitions of Spearman's rho are possible with an arbitrary dimension $d$: see Schmidt and Schmid (2007), for instance.
Note that, when $d=2$, $\rho(\beta,y)$ is simply the correlation between $F_1(X_1 |\ZZ)$ and $F_2(X_2 |\ZZ)$ given $\beta'\ZZ$.
Therefore, it can be estimated relatively easily, at least when the dimension of $\ZZ$ is ``reasonable''.

\mds

{\bf\it Example 3: Kendall's tau.}
To fix the ideas, let us assume $d=2$.
The Kendall's tau of $\XX$ conditional on $\ZZ=\zb$ appears in several papers in the literature (see Veraverbeke et al., 2011 for instance).
Under the single-index assumption, it is written here
\beq
\tau_{\zb} = 4\int C(\uu | \zb) C(d\uu | \zb)-1 = 4\int C_\beta(\uu | \beta'\zb) C_\beta(d\uu |\, \beta'\zb)-1.
\label{defTau}
\eeq
Since it depends only on $\beta'\zb$, it is denoted by $\tau(\beta,\beta'\zb)$.
Then, managing Kendall's tau, we work under Assumption (A1) usually.
The parameter $\beta$ and then $\psi(\beta,\beta'\zb)$ can be estimated empirically, replacing $C_\beta(\cdot |\beta'\zb)$ by an empirical
counterpart in the previous integral.

\mds

If $(\XX,\ZZ)$ and $(\YY,\ZZ)$ denote independent copies knowing $\ZZ$, note that
\begin{eqnarray*}
\lefteqn{ E[\1( X_1 > Y_1, X_2 > Y_2) |\,\beta'\ZZ=y   ] }\\
&=& E[E[\1\left( F_1(X_1 | \ZZ) > F_1(Y_1 |\ZZ), F_2(X_2 |\ZZ) > F_2(Y_2 |\ZZ) \right) | \ZZ] |\,\beta'\ZZ=y   ]  \\
&=& E[\int  C(\uu | \ZZ)\, C(d\uu |\ZZ) | \beta'\ZZ=y]= \int  C_\beta(\uu | y) \, C_\beta(d\uu |\, y).  \hspace{3cm}
\end{eqnarray*}
This implies that the Kendall's tau of $\XX$ given $\beta'\ZZ=y$ is $\tau(\beta,y)$, under~(\ref{single_index_cop}).
Incidentally, we have proved that
$$ \int  C_\beta(\uu | y) \, C_\beta(d\uu | y) = \int  \tilde C_\beta(\uu | y) \, \tilde C_\beta(d\uu | y), \, \text{and}$$
\beq
\tau(\beta,\beta'\zb)=4  \int  \tilde C_\beta(\uu | y) \, \tilde C_\beta(d\uu | y)-1.
\label{tautilde}
\eeq

\mds

Moreover, since
$$ E[\1( X_1 > Y_1, X_2 > Y_2) |\beta'\ZZ=y   ] = \int H_\beta(\xx,+\infty | \beta'\ZZ=y)\, H_\beta(d\xx,+\infty |\beta'\ZZ=y),$$
we recognize Assumption (A2), and
\beq
\tau(\beta,\beta'\zb)=4  \int H_\beta(\xx,+\infty | \beta'\zb)\, H_\beta(d\xx,+\infty |\beta'\zb) -1.
\label{tauHbeta}
\eeq

In other terms, Kendall's tau are of the two types (A1) and (A2) simultaneously. And the relations~(\ref{tautilde}) and~(\ref{tauHbeta})
 are very useful in practice. Indeed, the estimation of $H_\beta(\cdot | y)$ or $\tilde C_\beta(\cdot | y)$ is less demanding than
the non parametric estimation of $C_\beta(\cdot | \beta'\zb)$: an empirical counterpart of $H_\beta(\xx | y)$ or $\tilde C_\beta(\uu | y)$ does not suffer from the curse of dimension because it necessitates only conditioning subsets in $\Rb$, contrary to $ C_\beta(\uu | y)$ that involves conditioning wrt $\zb \in \Rb^p$ to manage its marginal laws.

\mds

In dimension $d$, many Kendall's tau can be built.
Logically, these Kendall's tau may be associated to any couple of variables $(X_i,X_j)$, $i,j=1,\ldots,d$, $i\neq j$.
Or they can be defined formally as in~(\ref{defTau}), with $d$-dimension integrals, or even $d'$-dimension integrals, $d'<d$ if we focus on some sub-vectors of $\XX$.
Globally, all such quantities are linear function of $\int C(\uu_I,\1_{\bar I} | \zb) \,  C(d\uu_I, \1_{\bar I} | \zb)$, where $I$ is a subset of $\{1,\ldots,d\}$ and $\bar I$ is its complement~\footnote{Obviously, $\uu_I,\1_{\bar I}$ denotes a $d$-dimensional vector whose components are $u_k$ when $k\in I$, and are equal to one otherwise.}. These dependence measures are candidates to provide convenient moments.
Note the two usual generalizations of Kendall's tau in dimension $d$: the first one has been proposed by Joe (1990) as
\beq
\tau_d (\zb) := \frac{1}{2^d -1}\left\{ 2^d \int C(\uu | \zb) C(d\uu | \zb)-1 \right\},
\label{defTau_d}
\eeq
and the second one has been introduced by Kendall and Babington Smith (1940) as the average value of Kendall's tau over all possible couples $(X_k,X_l)$, $k,l=1,\ldots,d$, $k\neq l$.
See Genest et al. (2011) for details and complementary results.
In every case, the same arguments as above apply, providing straightforward $d$-dimensional extensions of~(\ref{tautilde}) and~(\ref{tauHbeta}).
As a consequence, such generalized Kendall's tau are of the two types (A1) and (A2) simultaneously. Moreover, they can be written as functionals of $\tilde C_\beta$ itself.

\mds

In practice, the underlying copulae often depend on a few parameters only, say one or two (Archimedean copulae, typically).
In  the latter case, their Kendall's tau and/or Spearman's rho are sufficient to identify the underlying copula parameters.
And there often exists an explicit one-to-one relationship between $\theta$ and the latter dependence measures.
But, obviously, other moments may be considered, particularly some functionals of the conditional copula functions only.

\mds

Now, let us specify our estimator $\hat\psi$. The simplest solution we
adopt is to invoke kernel-type regression functions.
Under (A1), we can replace simply the conditional copula $C_\beta(\cdot | \beta'\ZZ=y)$ by a consistent estimator $\hat{C}(\cdot | \beta'\ZZ=y)$.
Several candidates exist in the literature. Historically, Fermanian and Wegkamp (2012) were the first ones to propose a nearest neigbour estimator
 of conditional copulae. Gijbels et al. (2011) introduced other non-parametric estimates, including Nadaraya-Watson, Gasser-M\"{u}ller, etc.

\mds

Under (A2), for every
$\beta\in \Bc$ and $y\in \Rb$, set $\hat\psi(\beta,y):= \Psi(\hat H_\beta(\cdot | y))$, where
\beq
 \hat H_\beta( \xx,\zb | y)= \sum_{j=1}^n w_{\beta,j,n}(y)\1 (\XX_j\leq \xx,\ZZ_j\leq \zb ),
 \label{def_Hbeta}
\eeq
$$w_{\beta, j, n}(y)=K\left(\frac{\beta'\ZZ_j-y}{h_n}\right) / \sum_{l=1}^n K\left(\frac{\beta'\ZZ_l-y}{h_n}\right),$$
for some kernel function $K:\Rb \rightarrow \Rb$ and some bandwidth sequence
$(h_n)$, $h_n>0$. Hereafter, we remove the latter sub-index $n$, i.e. $h:=h_n$ simply for any bandwidth.

\mds

To satisfy Condition~(\ref{unif_psi}), we have to rely on the functional link between the parameter $\psi$ and the underlying distributions, as evaluated under
(A1) and/or (A2). This depends on the regularity of the corresponding functionals $\Psi$ and on the uniform distance between the conditional empirical cdfs' and true ones.

\mds

For instance, under (A2), assume $\Psi$ is Lipschitz, with a Lipschitz constant $\lambda$ (at least when $\beta \in \Bc$ and $\zb\in\Zc$, and then $\zb'\beta$ belongs to a real compact subset). For
such couples $(\beta,\zb)$, we have
$$|\hat{\psi}(\beta,\beta'\zb)-\psi(\beta,\beta'\zb) |\leq \lambda \| \hat H_\beta(\cdot | \beta'\zb) -H_\beta(\cdot | \beta'\zb) \|_\infty$$
Assuming $\hat H_\beta$ is given by~(\ref{def_Hbeta}) and applying Corollary 3 in Einmahl and Mason (2005), we obtain
$$\sup_i|\hat{\psi}(\beta,\beta'\ZZ_i)-\psi(\beta,\beta'\ZZ_i) |\omega_{i,n}
\leq \lambda\sup_i \| \hat H_\beta(\cdot | \beta'\ZZ_i) -H_\beta(\cdot | \beta'\ZZ_i) \|_\infty \omega_{i,n} \longrightarrow 0,$$
a.e. and uniformly wrt $\beta \in \Bc$. This is sufficient to satisfy~(\ref{unif_psi}).

\mds

Note that $\Psi$ is Lipschitz in the case of Kendall's tau. Indeed, through an integration by parts and for two cdfs' $H$ and $H'$ (for which $H$ or $H'$ is continuous), we observe that
\begin{eqnarray*}
\lefteqn{ |\int H(\cdot | \zb) \, dH(\cdot, |\zb) - \int H'(\cdot | \zb) \, dH'(\cdot, |\zb) | }\\
&\leq &
|\int (H-H')(\cdot | \zb) \, dH(\cdot, |\zb) | + | \int (H-H')(\cdot | \zb) \, dH'(\cdot, |\zb) |  \\
&\leq &
\|(H-H')(\cdot | \zb)\|_\infty \cdot \left( \int |dH|(\cdot, |\zb) +  \int |dH'|(\cdot, |\zb) \right)
\leq 2 \|(H-H')(\cdot | \zb)\|_\infty.
\end{eqnarray*}

More generally, under (A2) and if $\Psi$ is Hadamard differentiable, there exist continuous linear maps $\dot\Psi_i$ s.t.
\begin{eqnarray*}
\lefteqn{ \hat{\psi}(\beta,\beta'\ZZ_i)-\psi(\beta,\beta'\ZZ_i) = \Psi( \hat H_\beta(\cdot | \beta'\ZZ_i))
- \Psi( H_\beta(\cdot | \beta'\ZZ_i))  }\\
&=&
\dot\Psi_i( (\hat H-H)_\beta(\cdot | \beta'\ZZ_i))
+ o(\| (\hat H-H)_\beta(\cdot | \beta'\ZZ_i)\|) .
\end{eqnarray*}
Under some additional conditions (particularly on the $\dot\Psi_i$), we get typically the uniformity of the latter identity wrt $\ZZ_i\in \Zc$.
But, thanks to Theorem 3 in Einmahl and
Mason (2005), there exists a sequence of positive numbers $(a_n)$, $a_n\rightarrow 0$, s.t.
$$ a_n\sup_{\beta \in \Bc}\sup_{\zb \in \Zc}
\| (\hat H-H)_\beta(\cdot | \beta'\zb) \|_\infty \longrightarrow 0$$
a.e. when $n\rightarrow 0$.
The latter result is true uniformly wrt bandwith sequences $(h_n)$ s.t. $nh_n/\ln n >>1$ and $h_n\rightarrow 0$.
Therefore,~(\ref{unif_psi}) is usually satisfied when $\Psi$ is Hadamard differentiable.

\mds

Note that the uniform consistency of the conditional copula function, simultaneously wrt to its argument and the conditioning value, is
not available in the literature. Therefore, checking Condition~(\ref{unif_psi}) under (A1) is more difficult than under (A2).

\mds

\subsection{The choice of the pseudo-estimations $\hat\UU$}

By assumption, $\beta$ is the index of the underlying
dependence functions (copulae) only. Therefore, $\hat \UU_i$ does
not depend on $\beta$.
Now, let us discuss the possible choices for $\hat \UU_i$, $i=1,\ldots,n$.
Actually, in Section~\ref{AsNorm}, we consider a generic class of estimates s.t., for all $k=1,\ldots,d$,
\beq
 \hat F_{k}(x|\zb) - F_{k}(x|\zb) = \frac{1}{n} \sum_{j=1 }^n a_{k,n}(\XX_j,\ZZ_j,x,\zb) +
 r_n(x,\zb),
 \label{condHatF}
\eeq
for some sequence $(r_n(x,\zb))$ that tends to zero sufficiently quickly~\footnote{in particular to satisfy~(\ref{unif_uu})} uniformly in probability, and for some
particular functions $a_{k,n}$. Then, we set $\hat U_{i,k} :=  \hat
F_{k}(X_{i,k}|\ZZ_i)$, $i=1,\ldots,n$, $k=1,\ldots,d$.
A lot of estimators of $F_{k}$, and then of $\UU_{i}$, may
be built and satisfy~(\ref{condHatF}).

\mds

A first example of such estimates is given by parametric
marginal conditional distributions: for every $x$ and $\zb$,
$F_{k}(x|\zb)=G_{k,\theta_k(\zb)}(x) $, for some family of cdfs' $\Gc_k := \{ G_{k,\theta_k}, \; \theta_k \in \Theta_k\}$.
Since this model is parametric, the function $\theta_k$ depends on a vector of parameters $\eta_k\in \Rb^{m_k}$. With a light abuse of notations, set
$\theta_k(\zb)=\theta_k(\zb,\eta_k)$, and $\theta_k(\cdot, \eta)$ is known for every $\eta$.
Assume we have found a consistent and asymptotically normal estimate $\hat\eta_k$, and set
$\hat F_{k}(x|\zb)=G_{k,\theta_k(\zb,\hat\eta_k)}(x) $.
This implies $\hat U_{i,k}=G_{k,\theta_k(\zb_i,\hat\eta_k)}(X_{i,k}) $.

\mds

Clearly, for every $i$, there exists $\theta_{k,i}^*$ and $\eta_k^*$ s.t.
$$|\hat U_{i,k} - U_{i,k} |\leq  |\nabla_\theta G_{k,\theta_{k,i}^*}(X_{i,k}) |\cdot |\partial_2 \theta_k (\ZZ_i,\eta_k^*) | \cdot |\hat\eta_k -\eta_k|,$$
where $| \theta_k(\ZZ_i,\eta_k) - \theta_{k,i}^* | \leq | \theta_k(\ZZ_i,\hat\eta_k) - \theta_k(\ZZ_i,\eta_k)|$ and
$|\eta_k - \eta_k^*|\leq |\hat\eta_k - \eta_k|$.

\mds

Typically, if $ \sup_{\theta_k} |\nabla_\theta G_{k,\theta_{k}}(X_{i,k}) |$ and
$\sup_{\eta_k}|\partial_2 \theta_k (\ZZ_i,\eta_k) |$ are bounded in probability (integrable, in particular), the condition~(\ref{unif_uu}) is satisfied, even without trimming.

\mds

Moreover, in a lot of usual cases (M-estimates, e.g.), it
can be checked by a limited expansion that the functions $\hat
F_{k}(x|\zb) $
satisfy~(\ref{condHatF}). Typically, for asymptotically normal estimators, we observe $n r_n(x,\zb)=O_P(1)$, and this result may be uniform under some conditions of regularity
concerning $G$ and $\theta_k(\cdot)$.
Such a choice of conditional margins induces the so-called estimator $\hat\beta^{(1)}$.

\mds

A second candidate is provided by nonparametric estimates of
conditional expectations. A usual kernel-based nonparametric
estimator of $ F(\cdot|\zb)$ on $\Rb^d$ is given by
\beq
\hat{F}(\xx|\zb)=\sum_{j=1}^n w_{j,n}(\zb)\1 (\XX_j\leq \xx),
\label{KernelF}
\eeq
with the weights
\beq
w_{ j, n}(\zb)=\KK\left(\ZZ_j-\zb, \hh\right) / \sum_{l=1}^n \KK\left(\ZZ_l-\zb,\hh\right),
\label{my_weights}
\eeq
where $\KK$ is a multivariate kernel and $\hh:=(h_1,\ldots,h_p)$ is a $p$-vector of
bandwidths $h_k>0$. To simplify and w.l.o.g., we can restrict ourselves on products of $p$ univariate kernels $K_k$ i.e.
\beq
\KK\left(\ZZ_j-\zb, \hh\right) = \frac{1}{h_1\cdots h_p}\prod_{k=1}^p K_k\left( \frac{Z_{j,k}-z_k}{h_k}\right).
\label{usual_kernel}
\eeq
Therefore, some nonparametric estimators of every marginal conditional cdf
$F_{k}(\xx|\zb)$ are obtained by setting $\hat{F}_{k}(x|\zb)=\hat{F}(x,+\infty_{(-k)}|\zb)$. The marginal
``unfeasible'' observations are $U_{i,k}=F_{k}(X_{i,k}|\ZZ_i),$ and
their estimated versions are $\hat{U}_{i,k}=\hat{F}_{k}(X_{i,k}|\ZZ_i)$.
In this case, it can be verified that~(\ref{condHatF}) is satisfied.

\begin{lemma}
\label{my_Fk}
For $k=1,\ldots,d$, define $\hat F_k$ as
\beq
\hat{F}_k(x|\zb)=\sum_{j=1}^n w_{j,n}(\zb)\1 (X_{j,k}\leq x), \label{KernelFk}
\eeq
with the weights given by~(\ref{usual_kernel}). Assume
\begin{itemize}
\item $f_{\ZZ}$, the density of $\ZZ$, exists and is strictly positive on $\Zc$. Moreover, it is $s$-times continuously differentiable, $s\geq 2$.
\item for every real $x$ and every $k$, the function $h_k(x,\cdot):\zb\mapsto P(X_k \leq x | \ZZ=\zb)f_{\ZZ}(\zb)$, defined on $\Zc$, is $s$-times continuously differentiable. Moreover,
$$ \sup_{x \in \Rb}\sup_{\zb\in \Zc} | d^{s}_{\zb} h_k(x,\zb)| \; \text{is bounded}.$$
\item the underlying kernel $\KK(\cdot,\1)$ is continuous, bounded, of bounded variation, $\int \KK(\zb,\1)\,d\zb=1$ and it is compactly supported~\footnote{To be specific,
this kernel has to be ``regular'' in the sense of Einmahl and Mason (2005), i.e. it has to satisfy their assumptions $K.i-K.iv$.}. Moreover, it is a multivariate $s$-order kernel, i.e.
$$ \int\prod_{j=1}^p z_j^{\alpha_j} K(\zb,\1)\, d\zb =0,$$ for every $p$-uplet of integers $(\alpha_1,\ldots,\alpha_p)$ s.t. $\alpha_j\in \{1,\ldots,s-1\}$ for some index $j$.
\end{itemize}
Then, for any $k=1,\ldots,d$, we have
\beq
 \hat F_{k}(x|\zb) - F_{k}(x|\zb) = \frac{1}{n} \sum_{j=1 }^n a_{k,n}(\XX_j,\ZZ_j,x,\zb) +
 r_n(x,\zb),
 \label{DevAsymFhat}
 \eeq
\bqa
\lefteqn{a_{k,n}(\XX_j ,\ZZ_j,\xx,\zb) = \frac{1}{ f_{\ZZ}(\zb)}
\left( \KK\left(\ZZ_j-\zb,\hh\right) \1(X_{j,k}\leq x) - E[\KK\left(\ZZ_j-\zb,\hh\right) \1(X_{j,k}\leq x)]  \right.  }\\
&-& \left.
P(X_k\leq x | \ZZ=\zb) \{ \KK\left(\ZZ_j-\zb,\hh\right) -E[ \KK\left(\ZZ_j-\zb,\hh\right)]  \} \right),
\eqa
\beq
 \lim\sup_n   \min(u^2_{n,1},u_{n,2}) \sup_{x\in \Rb,\zb\in \Zc} |r_n (x,\zb) | \leq  C_1 \; \text{a.e., and}
 \label{DevAsymRn}
 \eeq
\beq
\lim\sup_n  \min(u_{n,1},u_{n,2})\sup_{x\in \Rb,\zb\in \Zc}| \hat F_{k}(x|\zb) - F_{k}(x|\zb)| \leq C_2\; \text{a.e.}
 \label{DevAsymFhat2}
 \eeq
for some positive constants $C_1$, $C_2$,
$$  u_{n,1}:=\left(  \frac{n \prod_{l=1}^p h_l }{\max(-\ln(\prod_{l=1}^p h_l),\ln \ln n)} \right)^{1/2},\;\text{and}\; u_{n,2}:= \frac{1}{\max_{l=1,\ldots,p} h_l^s}  \cdot$$

\end{lemma}

\begin{proof}
Equation~(\ref{DevAsymRn}) is deduced directly from Theorem 2 in Einmahl and Mason (2005).
Moreover, by straightforward calculations, we get
$$r_{n,k}(x,\zb)=r^{(1)}_{n,k}(x,\zb)+r^{(2)}_{n,k}(x,\zb),$$
$$ r^{(1)}_{n,k}(x,\zb) = \frac{E\hat h(x,\zb)(\hat g - E\hat g)^2 (\zb)}{(E\hat g)^2 \hat g (\zb)}-\frac{(\hat h - E\hat h)(x,\zb)(\hat g - E\hat g)(\zb)}{\hat g(\zb) E\hat g (\zb)},$$
$$ r^{(2)}_{n,k}(x,\zb) = \frac{E\hat h (x,\zb) }{ E\hat g (\zb)}-F(x_k|\zb),$$
$$ \hat h(x,\zb)= \frac{1}{n} \sum_{j=1}^n \KK\left(\ZZ_j-\zb, \hh\right).\1(X_{j,k} \leq x) ,\;
 \hat g(\zb)= \frac{1}{n} \sum_{j=1}^n \KK\left(\ZZ_j-\zb, \hh\right),$$
that tends typically to $g=f_{\ZZ}$ and $h_k(x,\zb)=P(X_k\leq x | \ZZ=\zb) g(\zb)$.
By invoking the equations (3.7) and (3.8) in the proof of Theorem 2 in Einmahl and Mason (2005), we get the uniform convergence of $\hat h$ (resp. $\hat g$) towards
$E\hat h$ (resp. $E\hat g$) almost surely, at the same rate $u_n$.
Note their remark 8 justifies the choice of different bandwidths for every component of $\ZZ$.

\mds

Moreover, by usual limited expansion of $E\hat g - g$ and $E\hat h - h$, we can deal with the bias term. Due to our assumptions concerning the order of the kernel $\KK$ and the regularity conditions on the underlying laws, we obtain easily that
$r_{n,k}^{(2)} (x,\zb) = O(\max_{l=1,\ldots,p} h_l^s)$, providing the result.
\end{proof}

\mds

As a consequence, the condition~(\ref{unif_uu}) is satisfied for the nonparametric versions on $\hat\UU_i$ and for a wide range of bandwiths.
Let us denote the associated estimator by $\hat\beta^{(2)}$.

\mds

Between the two previous polar cases, there exist a lot of
candidates. For instance, to avoid the curse of dimension, it
may be assumed that any marginal conditional distribution depends on a reduced subset of $\ZZ$-components.
This often corresponds to intuition. For instance, imagine we want to describe the co-movements of corporate default rates in the USA and in Europe. A set of macro-economic variables (GDP, inflation, unemployment, etc) for both economic areas will be our explanatory variable $\ZZ$.
Then, it makes sense to assume that default rates in the USA (resp. in Europe) depend on American (resp. European) macro-factors only, but that the conditional copula between both default rates depends on the whole set of macro-factors $\ZZ$.

\mds

Alternatively, it could be assumed that some conditional distributions $F_k(\cdot | \ZZ)$
are given by particular single-index models, but with some parameter $\beta_k\in \Rb^{m_k}$ that are different of $\beta$.
If the latter index $\beta_k$ is estimated consistently by
$\hat\beta_k$, then we can adapt easily the previous nonparametric kernel
estimator: for any real number $y$,
$$\hat{F}_{k,\hat\beta_k}(x|y)=\sum_{j=1}^n w_{\hat\beta_k,j,n}(y)\1 (X_{j,k}\leq x),$$
where
$$w_{\hat\beta_k, j, n}(y)=K\left(\frac{\hat\beta_k'\ZZ_j-y}{h}\right) / \sum_{l=1}^n K\left(\frac{\hat\beta_k'\ZZ_l-y}{h}\right),$$
for some kernel function $K:\Rb \rightarrow \Rb$ and some bandwidth
$h>0.$ Obviously, $\hat{F}_{k,\hat\beta_k}(x|y)$ provides a nonparametric estimator of the cdf $F_{k,\beta_k}(x|y)$.
In this case, $U_{k,\zb}=F_{k}(X_{k}|\beta_k'\zb).$ To deal with pseudo-observations, we set
$U_{i,k,\beta_k}=F_{k,\beta_k}(X_{i,k}|\beta_k'\ZZ_i),$
and
$\hat{U}_{i,k}=\hat{F}_{k,\hat\beta_k}(X_{i,k}|\hat\beta_k'\ZZ_i)$.
For some conditions of regularity,~(\ref{condHatF}) can
be verified, see for example Du and Akritas (2002) for such a representation in the more general case where censored data is present. When all margins are assumed single-index, let us
denote by $\hat\beta^{(3)}$ the corresponding $\beta$ estimator.

\mds


Now, let us verify the conditions of Theorem~\ref{Th_consistency}, particularly Assumption~\ref{A_unifapprox}, in some particular cases.

\subsection{Examples}
\label{subs_examples}
Let us illustrate the previous ideas with a few standard copula models.

\mds

{\bf\it Example 4: the Gaussian copula}

Let us consider a $d$-dimensional conditional copula model: for every $\uu$ and $\zb$ and with usual notations, the true underlying copula is
$$ C_{\beta_0}(\uu | \ZZ=\zb)= C^G_{\Sigma(\zb)}(\uu)=\Phi_{\Sigma (\zb)} \left( \Phi^{-1}(u_1),\ldots,\Phi^{-1}(u_d)\right),$$
where the correlation matrix $\Sigma(\zb)=[\theta_{k,l}(\zb)]_{1\leq k,l \leq d}$ depends on the index $\beta_0'\zb$ only. With our previous notations,
$ \Sigma(\zb)=\psi(\beta_0,\beta'_0\zb)$. It is well-known that every component $\theta_{k,l}(\zb)$ of $\Sigma(\zb)$ is a function of a Kendall's tau:
$ \theta_{k,l}(\zb)= \sin(\pi \tau_{k,l}(\beta_0,\beta'_0\zb)/2),$ the conditional Kendall's tau that is associated to $(X_k,X_l)$, knowing $\beta_0'\ZZ=\beta'_0\zb$.
The latter quantity can be estimated by standard nonparametric techniques, and then
$$\hat \psi(\beta,\beta'\zb)=\left[\sin(\frac{\pi}{2} \hat\tau_{k,l}(\beta,\beta'\zb)) \right].$$
To be specific, we can choose
$$ \hat\tau_{k,l}(\beta, y) := 4 \int \hat C_{k,l}(u,v | \beta'\ZZ=y)\, \hat C_{k,l}(du,dv | \beta'\ZZ=y) - 1,$$
for some estimator $\hat C_{k,l}(\cdot| \beta'\ZZ=y)$ of the conditional copula of $(X_k,X_l)$ given $\beta'\ZZ=y$.
Alternatively, we can invoke an asymptotically equivalent estimator
$$ \hat\tau_{k,l}(\beta,\beta'\zb) :=4 \sum_{i=1}^n  \sum_{j=1}^n w_{i,h}(\beta'\zb)w_{j,h}(\beta'\zb) \1(X_{k,i} < X_{k,j},X_{l,i} < X_{l,j})-1,$$
for some weights, for instance the standard Nadaraya-Watson kernel
$$ w_{i,h}(y):=K\left( \frac{y- \beta'\ZZ_i}{h} \right) / \sum_{l=1}^n K\left( \frac{y- \beta'\ZZ_l}{h} \right).$$
See Gijbels et al. (2011) for alternative weights and estimators.

\mds

Once we have stated $\hat\psi$, it remains to set the marginal cdfs' $\hat U_k$, $k=1,\ldots,d$, to be able to calculate our estimator $\hat\beta$. To fix the ideas, we rely on the
standard univariate kernel-based conditional distributions, as given in~(\ref{my_weights}): $\hat U_{i,k}:=\hat F(X_{i,k} |\ZZ_i)$ and our estimator is then $\hat\beta^{(2)}$.

\mds

Concerning Assumption~\ref{A_identif}, the only thing to be checked is~(\ref{A_UnifLLN}).
This is guaranteed when the random matrix $\Sigma^{-1}(\ZZ)$ is staying ``under control'', for instance when all eigenvalues of $\Sigma(\ZZ)$ are uniformly bounded from below almost surely.
It is sufficient to assume that
\begin{equation}
\sup_{\zb\in \Zc}\sup_{\beta\in \Bc} \lambda_{\min}(\psi(\beta,\beta'\zb)) \geq \underline{\lambda} >0,
\label{theta_max}
\end{equation}
where $\lambda_{\min}(\Sigma)$ denotes the smallest eigenvalue of any nonnegative matrix $\Sigma$.
In this case, it is easy to bound the log-density of $\XX$ (conditional on $\ZZ$) from above, and to satisfy~(\ref{A_UnifLLN}).

\mds

Assumption~\ref{A_unifapprox} is the most tricky one. In Section~\ref{con_tau}, some sufficient conditions are given to satisfy~(\ref{unif_psi}).
It remains to verify~(\ref{unif_uu}). We can apply our Lemma~\ref{my_Fk}: under its conditions and
if all the bandwidths we consider in $\hat\UU_i$ behave as the same power of $n$, say $n^{-\pi}$ (the usual case), there exists a constant $C$ s.t.,
with probability one,
$$ \lim\sup_n\sup_{i=1,\ldots,n} | \hat \UU_i - \UU_i| .\1(\ZZ_i \in \Zc)/\delta_n \leq C,\; \text{where} \; \delta_n:=\sqrt{\ln(n)} n^{-(1-p\pi)/2} + n^{-\pi s}.$$
Note that, for consistency purpose, we can choose any $\pi$ s.t. $\pi < 1/p$.
And $\nu_n$ can be chosen arbitrarily as long as we have $\nu_n >> \delta_n$, and then the condition~(\ref{unif_uu}) is satisfied.

\mds

Assumption~\ref{A_moments} is satisfied for the Gaussian copula, as in most usual copula families.
In our case and under~(\ref{theta_max}), we choose $r(\uu)\propto \sum_{k=1}^d \left(\Phi^{-1}(u_k)\right)^2$, and
$ \tilde r_k(\uu) \propto  \Phi^{-1}(u_k) \sum_{l=1, l\neq k }^d \left(\Phi^{-1}(u_l)\right) .\left( \phi \circ \Phi^{-1}(u_k) \right)^{-1}$.

\mds

Therefore, the estimator $\hat\beta^{(2)}$ is consistent under the Gaussian copula framework.

\mds

{\bf\it Example 5: the Clayton copula}

The Clayton copula is often useful in finance, because it induces left tail dependence, a common feature of asset returns.
When the values of its parameter are strictly positive, the conditional Clayton copula is written
$$ C(\uu | \zb)= \left( \sum_{k=1}^d u_k^{-\theta(\zb)} - d+1 \right)^{-1/\theta(\zb)},\;\uu\in (0,1)^d,$$
with $\theta(\zb)=\psi(\beta,\beta'\zb)$ under the single-index assumption.
As with the Gaussian copula model, we can evaluate $\hat\psi$ with conditional Kendall's tau, because of their one-to-one mapping.
Indeed, invoking Example 1 in Genest et al. (2011), the Kendall tau of a Clayton model is equal to
$$\tau_d=  \frac{1}{2^{d-1}-1} \left\{ -1 + 2^d \prod_{k=0}^{d-1} \left(\frac{1+k \theta }{2+ k\theta} \right)    \right\}.$$
It is easy to check that the latter mapping between $\tau$ and $\theta$ is one-to-one.
The density of the Clayton copula with parameter $\theta >0$ is given by
$$ \ln c_\theta(\uu |\zb)= \sum_{k=1}^{d-1} \ln(1+k\theta) - (\theta +1)\sum_{k=1}^d \ln (u_k) - \left( \frac{1}{\theta} +d\right)
\ln\left( \sum_{k=1}^d u_k^{-\theta} -1 \right).$$

Assume that there exists $\underline{\theta}$ and $\overline{\theta}$ s.t., for every $\zb\in \Zc$ and every $\beta\in \Bc$,
$ \underline{\theta} \leq \psi(\beta,\beta'\zb) \leq \overline{\theta}$.
Then Assumption~\ref{A_identif} is satisfied. Indeed, note that
$$0\leq \ln\left( \sum_{k=1}^d u_k^{-\theta} -d+1 \right) \leq  \sum_{k=1}^d \ln\left( d u_k^{-\theta} \right)
\leq d\ln(d) -  \overline{\theta} \sum_{k=1}^d \ln\left( u_k \right) .$$
Denoting $V$ a r.v. that is uniform on $(0,1)$, we have
$$E[ \ln (F_k(X_k |\ZZ))] = E_{\ZZ} \left[ E_{X_k |\ZZ}[  \ln (F_k(X_k |\ZZ))   |\ZZ ]\right] =
E_{\ZZ} \left[ E_{X_k |\ZZ}[  \ln V ]\right] =(-1),$$
and~(\ref{A_UnifLLN}) follows.

\mds

Assumption~\ref{A_unifapprox} is satisfied with the same arguments as for the Gaussian copula.
Assumption~\ref{A_moments} can be verified relatively easily. Concerning $\nabla_\theta \ln c_\theta(\uu |\zb)$, the relevant reproducing u-shaped function is given by the product of the functions $r_k(u) \propto -\ln(u_k) \1(u_k\in (0,1/2]) -\ln (1-u_k) \1(u_k\in(1/2,1))$, $k=1,\ldots,d$. To see this, use the following inequality: for every $\uu\in (0,1)^d$,
$$ \frac{ |\sum_{k=1}^d u_k^{-\theta} \ln u_k | }{\sum_{k=1}^d u_k^{-\theta} - d+1} \leq
\max_k u_k^{-\theta} \cdot \frac{ \sum_{k=1}^d |\ln u_k| }{\sum_{k=1}^d u_k^{-\theta} - d+1} \leq  -\sum_{k=1}^d  \ln u_k.$$
To manage $\nabla_{u_k} \ln c_\theta(\uu |\zb)$, the relevant reproducing u-shaped function is obtained by replacing $r_k$ above by $\bar r_k(u)\propto u_k^{-1}(1-u_k)^{-1}$.
Assumption~\ref{A_moments} follows by setting $\tilde r_k(\uu) = \bar r_k (u_k) \prod_{l\neq k} r_l(u_l)$.

\mds

{\bf\it Example 6: the Gumbel copula}

The $d$-dimensional Gumbel copula is given by
$$ C_{\theta}(\uu):= \exp\left( - \left[ \sum_{k=1}^d (-\ln u_k)^{\theta} \right]^{-1/\theta}       \right),$$
for some parameter $\theta >1$.
It exhibits right tail dependence.
\mds

Its Kendall's tau in dimension $d$, as defined by~(\ref{defTau_d}) has been calculated in Genest et al. (2011):
$$ \tau_d = \frac{1}{2^d - 1} \left[-1 + 2^d \sum \Cc_{\vec{m}} \frac{(m-1)!}{(d-1)!}\left(\frac{1}{2\theta} \right)^{m-1}
\prod_{q=1}^d \left( \prod_{l=1}^{q-1} (k-1/\theta)    \right)^{m_q}      \right],$$
where $\vec{m}:=(m_1,\ldots,m_d)$, $m=m_1+\ldots + m_d$, and the summation is taken over all $d$-uplets of integers s.t.
 $m_1+2m_2+\ldots, + dm_d=d$. For every $\vec{m}$, $\Cc_{\vec{m}}$ denotes a positive constant.
But note that
$$\left(\frac{1}{\theta} \right)^{m-1}
\prod_{q=1}^d \left( \prod_{l=1}^{q-1} (k-1/\theta)    \right)^{m_q}=
\left(\frac{1}{\theta} \right)^{d-1}
\prod_{q=2}^d \left( \prod_{l=1}^{q-1} (k\theta-1    \right)^{m_q}:=\chi_{\vec{m}}(\theta),$$
and
\bqa
 (\ln \chi_{\vec{m}})'(\theta) &\propto &
 -(d-1) + \sum_{q=2}^d \sum_{k=1}^{q-1} \frac{k m_q}{k-1/\theta}   \\
&> &
-(d-1) + \sum_{q=2}^d \sum_{k=1}^{q-1} m_q =0.
\eqa
Therefore, every function $\chi_{\vec{m}}$ above is invertible, and the mapping between $\theta$ and $\tau $ is one-to-one, as usual. We can use the empirical (conditional)
Kendall's tau to evaluate the under parameter $\theta$ (or $\theta(\zb)$ more generally).

\mds

The Gumbel copula density is a linear combination of the functions
$$ c_j(\uu):= C_\theta(\uu) \left[ \sum_{k=1}^d (-\ln u_k)^{\theta} \right]^{j/\theta - d} \prod_{k=1}^d \frac{(-\ln u_k)^{\theta-1}}{u_k},$$
for some $j=1,\ldots,d$.
In the single-index model, $\theta$ is a function of $\zb$. Assume that $\theta(\zb)$ belongs to a fixed interval
$[\underline{\theta},\overline{\theta}]\subset ]1,+\infty[$ almost everywhere. Therefore, the density $c_{\theta(\zb)}$ of a Gumbel copula satisfies
$$ c_{\theta(\zb)}(\uu) \leq Cst. C(\uu)
\max_{\theta \in \{ \underline{\theta},\overline{\theta}\}}
\left\{ \left[ \sum_{k=1}^d (-\ln u_k)^{\theta} \right]^{j/\theta - d} \prod_{k=1}^d \frac{(-\ln u_k)^{\theta-1}}{u_k} \right\},$$
for every $\uu\in (0,1)^d$ and some constant $Cst$.
By taking the logarithm of the previous r.h.s., it is easy to check that~(\ref{A_UnifLLN}), and then Assumption~\ref{A_identif}, are satisfied.

\mds

Assumption~\ref{A_unifapprox} is satisfied with the same arguments as above.
After lengthly calculations, we can check Assumption~\ref{A_moments} too, by noticing that
$$ \sup_{\theta \in [\underline{\theta},\overline{\theta}]} | \partial{u_k} c_{\theta} (\uu) |
\leq Cst. h_k(\uu) C_\theta(\uu)/u_k^2:= \tilde r_k(\uu),$$
for some slowly varying functions $h_k$ (deduced from the powers
of the functions $u_l \mapsto \ln u_l$, $l=1,\ldots,d$). The function $\tilde r_k$ belongs to $\Rc_d$ since $C_\theta(\uu)$ behaves as $u_k$ when $u_k$ tends to zero. Therefore
$E[U_k(1-U_k)\tilde r_k(\UU)]<\infty$.

\section{Asymptotic normality}
\label{AsNorm}
\subsection{Notations and assumptions}

For convenience, denote $\psi_i = \psi(\beta_0,\beta_0'\ZZ_i)$
and $\hat\psi_i = \hat\psi(\beta_0,\beta_0'\ZZ_i)$.

\mds

Introduce the set of indicator functions
\bqa
\lefteqn{
\mathcal{H}=\left\{g:[0,1]^d\times \Rb^p \rightarrow [0,1], (\uu,\zb)\mapsto \1(\uu\in B_{\aaa,\bb},\zb \in \tilde B_{\ccc,\dd}),  \right.}\\
& \text{for some} & \left. B_{\aaa,\bb}:=\prod_{k=1}^d [a_k,b_k]\subset [0,1]^d \; \text{and} \;
\tilde B_{\ccc,\dd}:=\prod_{k=1}^p [c_k,d_k] \subset \Rb^p \right\}.
\eqa
Since all the subsets we consider in $\Hc$ are boxes, it is simple to verify that $\Hc$ is universally Donsker (for instance, see Example 2.6.1 and apply Lemma 2.6.17 in van der Vaart and Wellner (1996)).

\mds

\begin{assum}
\label{dini}
For every $\zb\in\Zc$, assume that $\psi_{\zb}:\Bc \rightarrow \Theta, \beta\mapsto \psi(\beta,\beta'\zb)$ is three times continuously differentiable.
Moreover, set $\ln c:(0,1)^d \times \Theta\rightarrow \Rb, (\uu,\theta) \mapsto \ln c_\theta(\uu)$. Assume that $\nabla_{\uu}\nabla^2_{\theta} \ln c_\theta(\uu)$ exists on
$(0,1)^d \times \Theta$.
\end{assum}

\begin{assum}
\label{d0}
Let the functions on $(0,1)^d \times \Zc$ defined by
$$f(\uu,\zb)=\frac{\nabla_{\theta}c_{\theta}}{c_{\theta}}_{|\theta=\psi(\beta_0,\beta_0'\zb)} (\uu),\;
\text{and}
\;\hat{f}(\uu,\zb)=\frac{\nabla_{\theta}c_{\theta}}{c_{\theta}}_{|\theta=\hat{\psi}(\beta_0,\beta_0'\zb)} (\uu).$$
For almost every realization, the functions $f$ and $\hat f$ belong to a Donsker class for the underlying law of $(\XX,\ZZ)$, that is denoted by $\Fc_1$.

\end{assum}

\begin{assum}
\label{d6} Let the functions on $\Zc$ defined by
$$p:\zb\mapsto p(\zb)=\nabla_{\beta}\psi(\beta,\beta'\zb)_{|\beta=\beta_0},\; \text{ and}\;
\hat{p}:\zb\mapsto \hat{p}(\zb)=\nabla_{\beta}\hat{\psi}(\beta,\beta'\zb)_{|\beta=\beta_0}.$$
For almost every realization, the functions $p$ and $\hat p$ belong to a Donsker class for the underlying law of $(\XX,\ZZ)$, that is denoted by $\Fc_2$.

\end{assum}

\begin{assum}
\label{hessc}
Assume that
$ E\left[ \sup_{\theta \in \Theta} |\nabla^2_\theta \ln c_\theta (\UU_{\ZZ}) |.\1(\ZZ \in \Zc)\right] < +\infty$.
Moreover, for every $(\uu,\uu')\in (0,1)^{2d}$, we have
\begin{eqnarray}
\left|\nabla_{\theta} \ln c_{\theta}(\uu)-\nabla_{\theta}\ln c_{\theta'}(\uu)\right| &\leq & \Phi(\uu).|\theta-\theta'|, \label{lip1}\\
\left|\nabla^2_{\theta} \ln c_{\theta}(\uu)-\nabla^2_{\theta} \ln c_{\theta'}(\uu)\right| &\leq & \Phi(\uu).|\theta-\theta' |, \label{lip2}
\end{eqnarray}
for some function $\Phi$ s.t. $E[\Phi(\UU)]<\infty$.
Moreover, there exists a function $r_3$ in $\Rc_d$ s.t., for every $\uu\in (0,1)^d$,
$$\sup_{\theta\in \Theta}\left| \nabla^3_{\theta}\ln c_{\theta}(\uu) \right| \leq r_3(\uu),\;E\left[ r_3(\UU_{\ZZ}) \1(\ZZ\in \Zc)\right]<\infty.$$

\mds

\end{assum}

\begin{assum}
\label{lipschitz}
Assume that, for every $(\beta_1,\beta_2)\in \Bc^2$ and $j=1,2$,
$$\sup_{\zb \in \Zc}|\nabla^j_{\beta}\psi(\beta_1,\beta_1'\zb)-\nabla^j_{\beta}\psi(\beta_2,\beta_2'\zb)|\leq C.|\beta_1-\beta_2|,$$
where $C$ is a finite constant.
\end{assum}

\begin{assum}
\label{hess}
Assume that
\begin{eqnarray}
\sup_{\beta\in \Bc,\zb\in \Zc}\left| \psi(\beta,\beta'\zb)-\hat{\psi}(\beta,\beta'\zb)\right| &=& o_P(1), \label{unifcons1} \\
\sup_{\beta\in \Bc,\zb\in \Zc}\left|\nabla_{\beta} \psi(\beta,\beta'\zb)-\nabla_{\beta} \hat{\psi}(\beta,\beta'\zb)\right| &=& o_P(1), \label{unifcons2} \\
\sup_{\beta\in \Bc,\zb\in \Zc}\left|\nabla_{\beta}^2 \psi(\beta,\beta'\zb)-\nabla_{\beta}^2 \hat{\psi}(\beta,\beta'\zb)\right| &=& o_P(1).\label{unifcons3}
\end{eqnarray}
\end{assum}

\begin{assum}
\label{uchapeau}
For every $k=1,\ldots,d$, there exists a function $\Gamma_k\in \Rc_d$ s.t.
$$\sup_{\theta \in \Theta}\left|\partial_{u_k}\nabla_{\theta}( \ln c_\theta)(\uu)\right|
+ \sup_{\theta \in \Theta}\left|\partial_{u_k}\nabla^2_{\theta}( \ln c_\theta)(\uu)\right| \leq \Gamma_k(\uu),$$
$$ E\left[U_k^{\alpha}(1-U_k)^{\alpha}\Gamma_k (\UU_{\ZZ}).\1(\ZZ\in \Zc) \right]<\infty,$$
for some $\alpha \in [0,1[$.
\end{assum}

\begin{assum}
\label{dbis}
 Assume that
\begin{eqnarray*}
\sup_{\zb\in\Zc}|\hat{\psi}(\beta_0,\beta_0'\zb)-\psi(\beta_0,\beta_0'\zb)| &=& O_P(\eta_{1n}), \\
\sup_{\zb\in\Zc}|\hat{p}(\zb)-p(\zb)| &=& O_P(\eta_{2n}),
\end{eqnarray*}
with $\delta_n^{1-\alpha}\eta_{jn}=o(n^{-1/2})$, for $j=1,2$, $\eta^2_{1n}=o(n^{-1/2})$, and $\eta_{1n}\eta_{2n}=o(n^{-1/2})$.
\end{assum}

\begin{assum}
\label{sigma}
Assume that $\beta\mapsto M(\beta)$ is twice continuously differentiable. Its Hessian matrix at point $\beta_0$
is denoted by $\Sigma=\nabla_\beta^2 M(\beta_0)$, and is invertible.
\end{assum}
Simple calculations provide
$$ \Sigma = \frac{1}{\Pb(\ZZ\in \Zc)} E\left[ \left( \nabla_{\theta}(\ln c_{\theta})_{|\theta= \psi_i} (\UU_i)
\nabla^2_{\beta} \psi(\beta,\beta'\ZZ_i) +\nabla^2_{\theta}(\ln c_{\theta})_{|\theta=\psi_i}(\UU_i) \nabla_{\beta} \psi_i\nabla_{\beta'} \psi_i
 \right) \cdot \1(\ZZ_i\in \Zc)    \right].$$


\begin{assum}
\label{ToManageTrimming}
For any $\uu\in \Rb^d$, set
$$ g(\uu,\zb):=\sup_{\theta \in B(\theta_0(\zb),\eta_{1,n})}
\sup_{\vv \in B(\uu,2\delta_n)} |\nabla_\theta \ln c_\theta (\vv) |,$$
where $B(\uu,\delta)$ (resp. $B(\theta,\eta)$) denotes the closed ball of center $\uu$ (resp. $\theta$) and radius $\delta$ (resp. $\eta$).
Assume
\beq
\label{eq_ToManageTrimming}
\sup_{k=1,\ldots,d}E[g (\UU_i,\ZZ_i)  \cdot \1(\ZZ_i\in \Zc, |U_{i,k}-\nu_n | < 2\delta_n) ] = o(n^{-1/2}),
 \eeq
and similarly after having replaced $\nu_n$ by $1-\nu_n.$
\end{assum}
The latter assumption is usually satisfied with a lot of usual copula models. Broadly speaking and when $c_\theta$ is continuous wrt its arguments and $\theta$ itself, it means that
\begin{equation*}
 \delta_n\int |\nabla_\theta c_\theta(\uu_{-k},\nu_n | \zb)_{|\theta=\theta_0(\zb)}| . \1(\zb\in \Zc) \, d\uu_{-k} \, d\Pb_{\ZZ}(\zb)= o(n^{-1/2}),
\end{equation*}
and the same replacing $\nu_n$ by $1-\nu_n$. Obviously, we denote by $(\uu_{-k},\nu_n)$ the $d$-dimensional vector whose components are $u_j$, when $j\neq k$, and whose $k$-th component is $\nu_n$.

\subsection{Main results}

\begin{theorem}
\label{asr}
Under Assumptions \ref{A_trim} to \ref{ToManageTrimming},
\begin{eqnarray*}(\hat{\beta}-\beta_0) &=& -\Sigma^{-1}\cdot \frac{1}{n}\sum_{i=1}^n \omega_{i,n} \frac{\nabla_{\theta}c_{\theta}}{c_{\theta}}_{|\theta=\psi_i} (\hat\UU_i)\nabla_{\beta} \psi(\beta,\beta'\ZZ_i)_{|\beta=\beta_0}+o_P(n^{-1/2}).
\end{eqnarray*}
\end{theorem}

\mds

\begin{proof}
By definition of $\hat{\beta},$
$\nabla_{\beta}M_n(\hat{\beta})=0.$ Next, a first order Taylor expansion leads to
$$-\nabla_{\beta}M_n(\beta_0)=(\hat{\beta}-\beta_0)\nabla^2_{\beta}M_n(\tilde{\beta}),$$
where $\tilde{\beta}=\beta_0+o_P(1)$, using the consistency of
$\hat{\beta}.$

From Lemma \ref{lemmahessian}, we have
$\nabla^2_{\beta}M_n(\tilde{\beta})=\nabla^2_{\beta}M(\tilde{\beta})+o_P(1).$ Moreover, from Assumption \ref{sigma} and the consistency of $\hat\beta$ (hence the consistency of $\tilde{\beta}$), we get $\nabla^2_{\beta}M_n(\tilde{\beta})=\Sigma+o_P(1).$

Next, we have
\begin{eqnarray*}
\nabla_{\beta}M_n(\beta_0) &=& \frac{1}{n}\sum_{i=1}^n \frac{\nabla_{\theta}c_{\theta}}{c_{\theta}}_{|\theta= \hat\psi_i} (\hat\UU_i)
\nabla_{\beta} \hat{\psi}(\beta,\beta'\ZZ_i)_{|\beta=\beta_0} \hat\omega_{i,n}.
\end{eqnarray*}

\mds

\textbf{a. Switch from the trimming functions $\hat\omega_{i,n}$ to $\omega_{i,n}$.}

\mds

Under Assumptions~\ref{A_unifapprox} and~\ref{ToManageTrimming}, we can
apply Lemma \ref{tricky2} with the function
$$ \chi(\UU_i,\ZZ_i):= \sup_{\theta \in B_{i,\theta}} \sup_{\vv \in B_{i,d}} |\nabla_\theta \ln c_{\theta}(\vv)|\cdot \sup_{\beta\in\Bc}\sup_{\zb\in\Zc} |\nabla_\beta \psi(\beta,\beta'\zb)|, \;\; \text{and}$$
$$ B_{i,\theta} := B(\theta_0(\ZZ_i),\eta_{1n}),\;\;B_{i,d}:=B(\UU_i,2\delta_n).$$
This implies
\begin{eqnarray*}
\nabla_{\beta}M_n(\beta_0) &=& \frac{1}{n}\sum_{i=1}^n \frac{\nabla_{\theta}c_{\theta}}{c_{\theta}}_{|\theta= \hat\psi_i} (\hat\UU_i)
\nabla_{\beta} \hat{\psi}(\beta,\beta'\ZZ_i)_{|\beta=\beta_0} \omega_{i,n}+ o_P(n^{-1/2}).
\end{eqnarray*}

\mds

Now, decompose
$$\nabla_{\beta}M_n(\beta_0) = A_{1n}+ A_{2n}+ R_{1n} +R_{2n}+R_{3n},$$
where
\begin{eqnarray*}
A_{1n} &:=&
 \frac{1}{n}\sum_{i=1}^n \frac{\nabla_{\theta}c_{\theta}}{c_{\theta}}_{|\theta=\psi_i} (\UU_i)\nabla_{\beta} \psi(\beta,\beta'\ZZ_i)_{|\beta=\beta_0}\omega_{i,n}, \\
 A_{2n} &:=& \frac{1}{n}\sum_{i=1}^n \left\{\frac{\nabla_{\theta}c_{\theta}}{c_{\theta}}_{|\theta=\psi_i} (\hat{\UU}_i)-\frac{\nabla_{\theta}c_{\theta}}{c_{\theta}}_{|\theta=\psi_i} (\UU_i)\right\}\nabla_{\beta} \psi(\beta,\beta'\ZZ_i)_{|\beta=\beta_0}\omega_{i,n}, \\
 R_{1n} &:=& \frac{1}{n}\sum_{i=1}^n \frac{\nabla_{\theta}c_{\theta}}{c_{\theta}}_{|\theta=\psi_i} (\hat{\UU}_i)\left\{\nabla_{\beta} \hat{\psi}(\beta,\beta'\ZZ_i)_{|\beta=\beta_0}-\nabla_{\beta} \psi(\beta,\beta'\ZZ_i)_{|\beta=\beta_0}\right\}\omega_{i,n}, \\
 R_{2n} &:=& \frac{1}{n}\sum_{i=1}^n \left\{\frac{\nabla_{\theta}c_{\theta}}{c_{\theta}}_{|\theta=\hat{\psi}_i} (\hat{\UU}_i)-\frac{\nabla_{\theta}c_{\theta}}{c_{\theta}}_{|\theta=\psi_i} (\hat{\UU}_i)\right\}\nabla_{\beta} \psi(\beta,\beta'\ZZ_i)_{|\beta=\beta_0}\omega_{i,n}, \\
  R_{3n} &:=& \frac{1}{n}\sum_{i=1}^n\left\{\frac{\nabla_{\theta}c_{\theta}}{c_{\theta}}_{|\theta=\hat{\psi}_i} (\hat{\UU}_i)-\frac{\nabla_{\theta}c_{\theta}}{c_{\theta}}_{|\theta=\psi_i} (\hat{\UU}_i)\right\}\left\{\nabla_{\beta} \hat{\psi}(\beta,\beta'\ZZ_i)_{|\beta=\beta_0}-\nabla_{\beta} \psi(\beta,\beta'\ZZ_i)_{|\beta=\beta_0}\right\}\omega_{i,n}.
\end{eqnarray*}
In this decomposition, we show that only the first two terms ($A_{1n}$ and $A_{2n}$) matter, and that
the $R_{jn}$, $j=1,2,3$, are $o_P(n^{-1/2})$.

\mds

\textbf{b. Study of $R_{1n}.$}

First observe that
$$R_{1n}=\frac{1}{n}\sum_{i=1}^n \frac{\nabla_{\theta}c_{\theta}}{c_{\theta}}_{|\theta=\psi_i} (\UU_i)\{\hat{p}(\ZZ_i)-p(\ZZ_i)\}\omega_{i,n}+R'_{1n},$$
$$R'_{1n}=\frac{1}{n}\sum_{i=1}^n \left\{ \frac{\nabla_{\theta}c_{\theta}}{c_{\theta}}_{|\theta=\psi_i} (\hat\UU_i) -
\frac{\nabla_{\theta}c_{\theta}}{c_{\theta}}_{|\theta=\psi_i} (\UU_i) \right\}
\{\hat{p}(\ZZ_i)-p(\ZZ_i)\}\omega_{i,n}.$$
By a limited expansion, we have
$$R'_{1n}=\frac{1}{n}\sum_{i=1}^n \left( \nabla_{\uu} \nabla_{\theta} (\ln c_{\theta})_{|\theta=\psi_i} (\UU_i^*).(\hat\UU_i - \UU_i) \right)
\{\hat{p}(\ZZ_i)-p(\ZZ_i)\}\omega_{i,n},$$
for some $\UU_i^*$ s.t. $|\UU_i^* - \UU_i| < |\hat \UU_i - \UU_i|$.
Reasoning as in the proof of Theorem~\ref{Th_consistency}, we can write
\bqa
\lefteqn{\Pb\left( n^{1/2}|R'_{1n}| > \ep  \right) \leq  \Pb(\sup_i |\hat\UU_i - \UU_i|> 2\delta_n,\ZZ_i\in \Zc)     }\\
&+& \Pb\left(  \frac{1}{n^{1/2}}\sum_{i=1}^n \left| \nabla^2_{\uu,\theta}  (\ln c_{\theta})_{|\theta=\psi_i} (\UU_i^*)\right|.|\hat\UU_i - \UU_i|
 \omega_{i,n}| \hat{p}(\ZZ_i)-p(\ZZ_i)|\cdot \1(|\hat\UU_i - \UU_i|\leq  2\delta_n)  >\ep  \right) \\
&\leq & \ep + \Pb\left(  \frac{\eta_{1n}}{n^{1/2}}\sum_{i=1}^n \sum_{k=1}^d\left| \Gamma_{k,1/2} (\UU_i)\right|.|\hat U_{i,k} - U_{i,k}|
 \omega_{i,n} \1(|\hat\UU_i - \UU_i|\leq  2\delta_n)  >\ep  \right)
\eqa
for $n$ large enough and invoking Assumption~\ref{dbis}.
Note that $$|\hat U_{i,k} - U_{i,k}|\1(|\hat\UU_i - \UU_i|\leq  2\delta_n)\omega_{i,n} \leq  C_\alpha U_{i,k}^{\alpha}(1-U_{i,k})^{\alpha}|\hat U_{i,k} - U_{i,k}|^{1-\alpha},\; \text{a.e.}$$
for some constant $C_\alpha$, when $n$ is sufficiently large, $i=1,\ldots,n$ and $k=1,\ldots,d$. This provides
\bqa
\Pb\left( n^{1/2}|R'_{1n}| > \ep  \right) & \leq & \ep + \Pb\left(  \frac{C'_\alpha\eta_{1n}\delta_n^{1-\alpha}}{n^{1/2}}\sum_{i=1}^n \sum_{k=1}^d  \Gamma_{k,1/2} (\UU_i) .
U^{\alpha}_{i,k}(1- U_{i,k})^{\alpha}
 \omega_{i}   >\ep  \right)  \\
 &\leq   & \ep + \frac{dC^{'}_\alpha n^{1/2}\eta_{1n}\delta_n^{1-\alpha}}{\ep} \sup_k E\left[ \Gamma_{k,1/2} (\UU_i) U^{\alpha}_{i,k}(1- U_{i,k})^{\alpha} \omega_{i}    \right]
\eqa
for some constant $C'_{\alpha}$.
Thanks to Assumption~\ref{uchapeau}, this means $\Pb\left( n^{1/2}|R'_{1n}| > \ep  \right) < 2\ep$ when $n$ is large enough, implying $R'_{1n}=o_P(n^{-1/2})$.

\mds

Moreover, with obvious notations, $R_{1n}$  can be rewritten as
$$R_{1n}=\frac{1}{n}\sum_{i=1}^n \left\{ \tilde{g}_n(\XX_i,\ZZ_i)-\tilde{g}(\XX_i,\ZZ_i)\right\}\omega_{i,n}+R'_{1n},$$
where $\tilde{g}_n$ and $\tilde{g}$ both belong to $\mathcal{F}_3=\Fc_1.\mathcal{F}_2.\mathcal{H},$ which is a Donsker class of functions. Indeed, the fact that $\mathcal{F}_3$ is a Donsker class follows from the permanence properties of Examples 2.10.10 and 2.10.7 in van der Vaart and Wellner (1996). Moreover, from Assumption \ref{dbis},
$$\sup_{\xx \in \Rb^d,\zb \in \Zc} |\tilde{g}_n(\xx,\zb)-\tilde g(\xx,\zb)|=o_P(1).$$ Therefore, the asymptotic equicontinuity of Donsker classes (see section 2.1.2 in van der Vaart and Wellner (1996) yields,
$$R_{1n}=\int \frac{\nabla_{\theta}c_{\theta}}{c_{\theta}}_{|\theta=\psi(\beta_0,\beta_0'\zb)}(\uu) \{\hat{p}(\zb)-p(\zb)\}\omega_{n}(\uu,\zb)d\mathbb{P}_{\left(\UU,\ZZ\right)}(\uu,\zb)+o_P(n^{-1/2}).$$

\mds

We can replace $\omega_n(\uu,\zb)$ above by $\1(\zb\in \Zc)$ if
$$\eta_{2n} \int |\nabla_{\theta} c_{\theta} (\uu)_{|\theta=\psi(\beta_0,\beta_0'\zb)} |
\cdot |\omega_{n}(\uu,\zb) - \omega_{\infty }(\uu,\zb)|\, d\uu \,d\mathbb{P}_{\ZZ}(\zb) = o(n^{-1/2}).$$
This is guaranteed under our assumption~\ref{ToManageTrimming}.

\mds

Then, under our assumptions, we can apply Fubini's theorem. This provides
\bqa
\lefteqn{\int \frac{\nabla_{\theta}c_{\theta}}{c_{\theta}}_{|\theta=\psi (\beta_0,\beta'_0\zb)} (\uu)\{\hat{p}(\zb)-p(\zb)\}\1(\zb\in\Zc)d\mathbb{P}_{\left(\UU,\ZZ\right)}(\uu,\zb) }\\
&=&  \int \{\hat{p}(\zb)-p(\zb)\} d\mathbb{P}_{\ZZ}(\zb) \left( \int \frac{\nabla_{\theta}c_{\theta}}{c_{\theta}}_{|\theta=\psi(\beta_0,\beta'_0\zb)} (\uu) \, \1(\zb\in\Zc)d\mathbb{P}_{\left(\UU |\ZZ=\zb\right)}(\uu) \right)  =0,
\eqa
by definition of $\psi (\beta_0,\beta_0'\zb),$ which maximizes $E[\ln c_{\theta}(\UU_{\zb})|\ZZ=\zb]$ with respect to $\theta$, for any $\zb\in \Zc$.
This shows that $R_{1n}=o_P(n^{-1/2}),$ and is therefore negligible.

\mds

\textbf{c. Study of $R_{2n}.$}

Write, from Assumption \ref{uchapeau} and with obvious notations,
\beq
R_{2n}= \frac{1}{n}\sum_{i=1}^n \left\{\frac{\nabla_{\theta}c_{\theta}}{c_{\theta}}_{|\theta=\hat{\psi}_i} (\UU_i)-\frac{\nabla_{\theta}c_{\theta}}{c_{\theta}}_{|\theta=\psi_i} (\UU_i)\right\}\nabla_{\beta} \psi(\beta,\beta'\ZZ_i)_{|\beta=\beta_0}\omega_{i,n} + R'_{2n},
\label{r2}
\eeq
where
\begin{eqnarray*}
R'_{2n} &=& \frac{1}{n}\sum_{i=1}^n \left\{\frac{\nabla_{\theta}c_{\theta}}{c_{\theta}}_{|\theta=\hat{\psi}_i} (\hat\UU_i)-
\frac{\nabla_{\theta}c_{\theta}}{c_{\theta}}_{|\theta=\psi_i} (\hat\UU_i)\right\}\nabla_{\beta} \psi(\beta,\beta'\ZZ_i)_{|\beta=\beta_0}\omega_{i,n} \\
&-&
\frac{1}{n}\sum_{i=1}^n \left\{\frac{\nabla_{\theta}c_{\theta}}{c_{\theta}}_{|\theta=\hat\psi_i} (\UU_i)-
\frac{\nabla_{\theta}c_{\theta}}{c_{\theta}}_{|\theta=\psi_i} (\UU_i)\right\}\nabla_{\beta} \psi(\beta,\beta'\ZZ_i)_{|\beta=\beta_0}\omega_{i,n} \\
&=& \frac{1}{n}\sum_{i=1}^n \left\{\nabla^2_{\theta} (\ln c_{\theta})_{|\theta=\psi_i} (\hat\UU_i)-
\nabla^2_{\theta} (\ln c_{\theta})_{|\theta=\psi_i} (\UU_i)\right\}.(\hat \psi_i -\psi_i) \nabla_{\beta} \psi(\beta,\beta'\ZZ_i)_{|\beta=\beta_0}\omega_{i,n} \\
&+& \frac{1}{2n}\sum_{i=1}^n \left\{\nabla^3_{\theta} (\ln c_{\theta})_{|\theta=\psi^*_i} (\hat\UU_i)-
\nabla^3_{\theta} (\ln c_{\theta})_{|\theta=\tilde\psi_i} (\UU_i)\right\}.(\hat \psi_i -\psi_i)^{(2)} \nabla_{\beta} \psi(\beta,\beta'\ZZ_i)_{|\beta=\beta_0}\omega_{i,n} \\
&=& \frac{1}{n}\sum_{i=1}^n \nabla_{\uu} \nabla^2_{\theta} (\ln c_{\theta})_{|\theta=\psi_i} (\UU^*_i).(\hat\UU_i - \UU_i)
.(\hat \psi_i -\psi_i) \nabla_{\beta} \psi(\beta,\beta'\ZZ_i)_{|\beta=\beta_0}\omega_{i,n} \\
&+& O_P\left( \sup_i|\hat \psi_i -\psi_i |^{2}  \right),
\end{eqnarray*}
for some $\UU_i^*$, $\psi^*_i$ and $\tilde\psi_i$ s.t. $|\UU_i^* - \UU_i|<|\hat \UU_i - \UU_i|$, $|\psi^*_i - \psi_i|<|\hat\psi_i - \psi_i|$
and $|\tilde\psi_i - \psi_i|<|\hat\psi_i - \psi_i|$. Note that we have invoked Assumption~\ref{hessc} to bound the last term on the r.h.s. in probability.
The main term on the r.h.s. is $O_P(\eta_{1n}\delta_n^{1-\alpha})=o_P(n^{-1/2})$ from Assumptions~\ref{uchapeau} and~\ref{dbis} (mimic the treatment of $R'_{1n}$ as above). We deduce
$ R'_{2n} = o_P(n^{-1/2})$.

\mds

Next, invoking assumptions~\ref{dbis} and~\ref{uchapeau}, the first term on the right-hand side of (\ref{r2}) can be rewritten as
$$\frac{1}{n}\sum_{i=1}^n \{h_n(\UU_i,\ZZ_i)-h(\UU_i,\ZZ_i)\}\omega_{i,n},$$
where
$\sup_{\uu,\zb}|h_n(\uu,\zb)-h(\uu,\zb)|=o_P(1),$ and $h_n$ and $h$ both belong to $\mathcal{F}_4=p.\mathcal{H}.\mathcal{F}_1,$ as a consequence of Assumption \ref{d0}. This is a Donsker class from Example 2.10.10 in van der Vaart and Wellner (1996). The asymptotic equicontinuity of the Donsker class $\mathcal{F}_4$ allows to write
\bqa
\lefteqn{ R_{2n}=\int \left\{\frac{\nabla_{\theta}c_{\theta}}{c_{\theta}}_{|\theta=\hat{\psi}(\beta_0,\beta'_0\zb)} (\uu)-\frac{\nabla_{\theta}c_{\theta}}{c_{\theta}}_{|\theta=\psi(\beta_0,\beta'_0\zb)} (\uu)\right\}\nabla_{\beta} \psi(\beta,\beta'\zb)_{|\beta=\beta_0}   }\\
&\cdot & \omega_{n}(\uu,\zb)d\mathbb{P}_{\left(\UU,\ZZ\right)}(\uu,\zb)+o_P(n^{-1/2}).\hspace{5cm}
\eqa
Decompose $\omega_{n}(\uu,\zb)$ as $\omega_{\nu}(\uu)\omega_{M}(\zb),$ where $\omega_{\nu}(\uu)=\mathbf{1}_{\min_k \min (1-u_k,u_k)\geq \nu_n},$ and
$\omega_M(\zb)=\mathbf{1}_{|\zb|\leq M}.$
The function
$$\phi_n(\zb)=\int \left\{\frac{\nabla_{\theta}c_{\theta}}{c_{\theta}}_{|\theta=\hat{\psi} (\beta_0,\beta'_0\zb)} (\uu)-\frac{\nabla_{\theta}c_{\theta}}{c_{\theta}}_{|\theta=\psi (\beta_0,\beta'_0\zb)} (\uu)\right\}\omega_{\nu}(\uu)d\mathbb{P}_{\left(\UU|\ZZ=\zb\right)}(\uu),$$
is a function of $\beta_0'\zb$ only. This is due to the fact that the distribution of $\UU$ given $\ZZ$ only depends on $\beta_0'\ZZ,$ due to the single-index assumption. With a slight abuse in notations, denote $\phi_n(\zb)=\phi_n(\beta_0'\zb).$ This leads to
$$R_{2n}=\int \phi_n(v) \left[\int \nabla_{\beta} \psi(\beta,\beta'\zb)_{|\beta=\beta_0} \omega_{M}(\zb)d\mathbb{P}_{\left(\ZZ|\beta_0'\ZZ\right)}(\zb|v)\right]d\mathbb{P}_{\beta_0'\ZZ}(v)+o_P(n^{-1/2}).$$
Next, as a consequence of Lemma \ref{tricky}, use that
$$\int \nabla_{\beta} \psi(\beta,\beta'\zb)_{|\beta=\beta_0} \omega_{M}(\zb)d\mathbb{P}_{\left(\ZZ|\beta_0'\ZZ=v\right)}(\zb)=0.$$
This implies $R_{2n}=o_P(n^{-1/2})$.

\mds

\textbf{d. Study of $R_{3n}.$}
By the same reasoning as for $R_{2n}$, we get
\begin{eqnarray*}
\lefteqn{ R_{3n}=\frac{1}{n}\sum_{i=1}^n \left\{\frac{\nabla_{\theta}c_{\theta}}{c_{\theta}}_{|\theta=\hat{\psi}_i} (\UU_i)-\frac{\nabla_{\theta}c_{\theta}}{c_{\theta}}_{|\theta=\psi_i} (\UU_i)\right\}   }\\
& \cdot  & \left\{ \nabla_{\beta} \hat\psi(\beta,\beta'\ZZ_i)_{|\beta=\beta_0} - \nabla_{\beta} \psi(\beta,\beta'\ZZ_i)_{|\beta=\beta_0} \right\}
\omega_{i,n}+o_P(n^{-1/2}).
\end{eqnarray*}
Due to Assumption \ref{dbis}
and Assumption \ref{hessc} (see equation (\ref{lip1})), we obtain $R_{3n}= o_P(n^{-1/2})$.
\end{proof}

\mds

Now, we need to introduce the way we estimate $\UU_i$ by pseudo-observations $\hat\UU_i$.
Therefore, additional assumptions are required to achieve asymptotic normality.

\begin{assum}
\label{my_hatF}
For every $k=1,\ldots,d$, $x \in \Rb$ and $\zb\in\Zc$, we can write
\beq
 \hat F_{k}(x|\zb) - F_{k}(x|\zb) = \frac{1}{n} \sum_{j=1 }^n a_{k,n}(\XX_j,\ZZ_j,x,\zb) +
 r_n(x,\zb),
 \label{condHatF2}
 \eeq
for some
particular functions $a_{k,n}$ and for some sequence $(r_n)$ s.t.
$$\sup_{x\in \Rb}\sup_{\zb\in \Zc} |r_n(x,\zb)| := r_{n,\infty} = o_P(n^{-1/2}).$$
\end{assum}

The latter assumption implies that, for every $i=1,\ldots,n$ and $k=1,\ldots,d$,
$$ \hat U_{i,k} - U_{i,k} =  \frac{1}{n} \sum_{j=1 }^n a_{k,n}(\XX_j,\ZZ_j,X_{i,k},\ZZ_i) +
 r_{n,i},\; n^{1/2}\sup_i |r_{n,i}|=o_P(1).  $$
Denote $\aaa_{n}(\XX_j,\ZZ_j,\XX_{i},\ZZ_i)$, or $\aaa_{i,j}$ even shorter, the $d$-vector whose components are $a_{k,n}(\XX_j,\ZZ_j,X_{i,k},\ZZ_i)$, $k=1,\ldots,d$.

\mds

In the case of the kernel-based estimates $\hat F_k$ of Lemma~\ref{my_Fk},
Assumption~\ref{my_hatF} is satisfied by using $s$-order kernels $\KK$ s.t. $\sup_k h_k=o(n^{-1/(2s)})$ and $n^{1/2} \prod_{k=1}^p h_k>> n^{a}$ for some $a>0$.
If $h_k=n^{-\pi}$ for all $k$, this necessitates $s> p$ and $\pi \in ]1/(2s);1/(2p)[$.

\begin{assum}
\label{bias}
Define $ \Lambda_{\psi(\beta_0,\beta_0'\zb)} := \nabla_{\uu}\nabla_{\theta} (\ln c_{\theta})_{|\theta = \psi(\beta_0,\beta_0'\zb)}$, and assume that
\beq
r_{n,\infty} E[ |\Lambda_{\psi(\beta_0,\beta_0'\ZZ_i)}(\UU_i)| \omega_{n}(\UU_i,\ZZ_i) ] =o(n^{-1/2}).
\label{Int_cond_nablau}
\eeq
Assume that there exists a function $W$ such that
$$\sup_{\xx\in \Rb^d,\zb\in \Zc}\left|E\left[a_{n}(\XX_j,\ZZ_j,\xx,\zb) \right]-W(\zb,\xx)\right|:=W_{n,\infty}=o(n^{-1/2}),$$
and such that
\beq
W_{n,\infty}E\left[\left|\Lambda_{\psi(\beta_0,\beta_0'\ZZ_i)}(\UU_i).W(\ZZ,\XX) .\nabla_{\beta}\psi(\beta,\beta'\ZZ_i)_{|\beta=\beta_0}\right| \omega_{i,n} \right]<\infty.
\label{Int_cond_nablau_W}
\eeq
\end{assum}

\mds

Choosing the kernel-based estimates $\hat F_k$ of Lemma~\ref{my_Fk}, we see that $ E\left[a_{n}(\XX_j,\ZZ_j,\xx,\zb) \right]=W(\zb,\xx)=0,$
and Assumption~\ref{bias} is automatically satisfied. This is most often the case with parametric marginal models too.

\mds

Moreover,~(\ref{Int_cond_nablau}) and~(\ref{Int_cond_nablau_W}) are often easily satisfied when
$E[ |\Lambda_{\psi(\beta_0,\beta_0'\ZZ_i)}(\UU_i)| .\1(\ZZ_i\in \Zc) ] <\infty.$ Note that the Gaussian copula model does not fulfill the latter condition. Nonetheless,
Assumption~\ref{bias} will be satisfied with a convenient choice of bandwidths, kernels and trimming sequences (see Subsection~\ref{Ex_contd}).

\begin{assum}
\label{deriv_u2}
For every $k=1,\ldots,d$, there exists a function $\zeta_k\in \Rc_d$ s.t.
$$\sup_{\theta \in \Theta}\left|\partial^2_{u_k}\nabla_{\theta}( \ln c_\theta)(\uu)\right| \leq \zeta_k(\uu),\; \text{and}$$
$$ E\left[U_k^{\gamma}(1-U_k)^{\gamma}\zeta_k (\UU_{\ZZ}).\1(\ZZ\in \Zc) \right]<\infty,$$
for some $\gamma \in [0,1]$. Moreover, $\delta_n^{2-\gamma}=o(n^{-1/2})$.
\end{assum}


\begin{assum}
\label{uprocess2}
Assume that
 $$v^2_n:=E\left[ | \aaa_{n}(\XX_j,\ZZ_j,\XX_{i},\ZZ_i) - E[ \aaa_{n}(\XX_j,\ZZ_j,\XX_{i},\ZZ_i)\,|\, \XX_i,\ZZ_i] |^2 \right] < \infty,$$
and that
$$ \frac{v_n^2}{n} E\left[\left|\Lambda_{\psi(\beta_0,\beta_0'\ZZ_i)}(\UU_i)\right|^2\omega_{i,n} \right]=o(1).$$
\end{assum}

\begin{corollary}
\label{ouf}
Under Assumptions \ref{A_trim} to \ref{uprocess2}, we have
$$n^{1/2}\left\{\Sigma.(\hat{\beta}-\beta_0)+b_n \right\}\Longrightarrow \mathcal{N}(0,S),$$
where $S=E[\omega_1\mathcal{M}_1\mathcal{M}_1'],$
where
$$\Mc_1=\frac{\nabla_{\theta}c_{\theta}}{c_{\theta}}_{|\theta=\psi_1} (\UU_1)\nabla_{\beta} \psi(\beta,\beta'\ZZ_1)_{|\beta=\beta_0}+
\Lambda_{\psi(\beta_0,\beta_0'\ZZ_1)}(\UU_1).W(\ZZ_1,\XX_1) \nabla_{\beta}\psi(\beta,\beta'\ZZ_1)_{|\beta=\beta_0},$$
$$ b_n = E[(\omega_{1,n}-\omega_i) \Mc_1]=E[\1(\UU_1\in [0,1]^d \ \Ec_n,\ZZ_1\in \Zc) \Mc_1].$$
Moreover, if
\bqan
\lefteqn{
E\left[   \Lambda_{\psi(\beta_0,\beta_0'\ZZ_1)}(\UU_1).W(\ZZ_1,\XX_1) \nabla_{\beta}\psi(\beta,\beta'\ZZ_1)_{|\beta=\beta_0} \right. \nonumber }\\
& \cdot & \left.
\left\{ \1( |U_{k,1} - \nu_n|<\delta_n) + \1( |1-U_{k,1} - \nu_n|<\delta_n) \right\} \right] = o(n^{-1/2}),
\label{cond_no_bias}
\eqan
for every $k=1,\ldots,d$, then $n^{1/2} b_n=o(1)$ and $n^{1/2}(\hat{\beta}-\beta_0) \Longrightarrow \mathcal{N}(0,\Sigma^{-1} S \Sigma^{-1}).$
\end{corollary}

Note that the bias $b_n$ cannot be removed in general, even if $E[\aaa_{i,j}]=0$.
Indeed, the trimming part $E[(\omega_{i,n}-\omega_i) \Mc_i]$ is of order $\delta_n$ typically, that has no reasons to be $o(n^{-1/2})$. To remove the asymptotic bias, we need~(\ref{cond_no_bias}). The latter condition is easily satisfied with purely parametric or nonparametric estimates, because $W(\ZZ,\XX)$ is zero or most often negligible in such cases.

\mds

\begin{proof}
We use the same notations as in the proof of Theorem \ref{asr}.
Recall that
$$A_{2,n}=\frac{1}{n}\sum_{i=1}^n\left\{\frac{\nabla_{\theta}c_{\theta}}{c_{\theta}}_{|\theta=\psi_i} (\hat{\UU}_i)-\frac{\nabla_{\theta}c_{\theta}}{c_{\theta}}_{|\theta=\psi_i} (\UU_i)\right\}\nabla_{\beta} \psi(\beta,\beta'\ZZ_i)_{|\beta=\beta_0}\omega_{i,n},$$
which can be rewritten as
\begin{eqnarray*} A_{2,n}&=&\frac{1}{n}\sum_{i=1}^n\Lambda_{\psi_i}(\UU_i) .[\hat{\UU}_i-\UU_i]\nabla_{\beta}\psi(\beta,\beta'\ZZ_i)_{|\beta=\beta_0}\omega_{i,n}
+O_P(\delta_n^{2-\gamma}) \\
&=:& A'_{2,n}+o_P(n^{-1/2}),
\end{eqnarray*}
thanks to  a limited expansion and invoking Assumptions~\ref{bias} and~\ref{deriv_u2}.
Next, under~(\ref{Int_cond_nablau}), we have
$$A'_{2,n}=\frac{1}{n^2}\sum_{j=1}^n \sum_{i=1}^n \Lambda_{\psi_i}(\UU_i).\aaa_{i,j} .\nabla_\beta \psi(\beta,\beta'\ZZ_i)_{|\beta=\beta_0}\omega_{i,n}+o_P(n^{-1/2}).$$
The leading term in $A'_{2,n}$ can be decomposed into $A'_{21}+A'_{22}$ where
\begin{eqnarray*}
A'_{21} &=& \frac{1}{n^2}\sum_{j=1}^n \sum_{i=1}^n \Lambda_{\psi_i}(\UU_i) .
E\left[\aaa_{i,j}|\ZZ_i,\XX_i\right]\nabla_{\beta}\psi(\beta,\beta'\ZZ_i)_{|\beta=\beta_0}\omega_{i,n}, \; \text{and}\\
A'_{22} &=& \frac{1}{n^2}\sum_{j=1}^n \sum_{i=1,i\neq j}^n \Lambda_{\psi_i}(\UU_i).
\left\{\aaa_{i,j} -E\left[ \aaa_{i,j}|\ZZ_i,\XX_i\right]\right\}.  \nabla_{\beta}\psi(\beta,\beta'\ZZ_i)_{|\beta=\beta_0}\omega_{i,n}.
\end{eqnarray*}
Due to Assumption \ref{bias}, Equation~(\ref{Int_cond_nablau_W}), we have
$$A'_{21}=\frac{1}{n}\sum_{i=1}^n \Lambda_{\psi_i}(\UU_i).W(\ZZ_i,\XX_i) \nabla_{\beta}\psi(\beta,\beta'\ZZ_i)_{|\beta=\beta_0}\omega_{i,n}+o_P(n^{-1/2}).$$
Next, observe that the main term of
$A'_{22}$ is of the form
$\sum_{i< j}\frak{U}(\ZZ_i,\XX_i,\ZZ_j,\XX_j)$, after symmetrization, where
$$E\left[\frak{U}(\ZZ_i,\XX_i,\ZZ_j,\XX_j)|\ZZ_j,\XX_j\right]=E\left[\frak{U}(\ZZ_i,\XX_i,\ZZ_j,\XX_j)|\ZZ_i,\XX_i\right]=0.$$
So, $A'_{22}$ is a degenerate $U-$process of order $2$.
It can be easily verified that its expectation is zero and
$$ Var(A'_{22}) =O\left(    \frac{v_n^2}{n^2} \cdot\int
| \Lambda_{\psi(\beta_0,\beta_0'\zb)}(\uu)|^2.
 |\nabla_{\beta}\psi(\beta,\beta'\zb)_{|\beta=\beta_0}|^2 \omega_{n}(\uu,\zb) \, d\Pb_{(\UU,\ZZ)}(\uu,\zb)
\right). $$
Under Assumptions~\ref{uprocess2}, we get $A'_{22}=o_P(n^{-1/2})$.


\mds

We have obtained
\begin{eqnarray*}
\lefteqn{
A_{1n}+A_{2n} =
\frac{1}{n}\sum_{i=1}^n \omega_{i,n} \frac{\nabla_{\theta}c_{\theta}}{c_{\theta}}_{|\theta=\psi_i} (\UU_i)\nabla_{\beta} \psi(\beta,\beta'\ZZ_i)_{|\beta=\beta_0} }\\
& + &  \frac{1}{n}\sum_{i=1}^n \omega_{i,n}\Lambda_{\psi_i}(\UU_i).W(\ZZ_i,\XX_i) \nabla_{\beta}\psi(\beta,\beta'\ZZ_i)_{|\beta=\beta_0}+o_P(n^{-1/2}) \\
& =: & n^{-1} \sum_{i=1}^n \omega_{i} \Mc_{i}+B_n+o_P(n^{-1/2}),
\end{eqnarray*}
by introducing a bias term $ B_n :=  n^{-1} \sum_{i=1}^n \{\omega_{i,n} -\omega_i\}\Mc_{i},$ due to the trimming procedure. Its expectation is denoted
by $ b_n=E[(\omega_{1,n}-\omega_i)\Mc_1] $, and its variance is $O(n^{-1}\delta_n)$.
The asymptotic bias is negligible under~(\ref{cond_no_bias}), by
recalling assumption~\ref{ToManageTrimming}, and then applying Lemma~\ref{tricky2}.

\mds
In every case, the result follows from a standard CLT, recalling the expansion of Theorem~\ref{asr}.
\end{proof}

\subsection{Examples cont'd}
\label{Ex_contd}
Let us check whether the conditions above apply to get the asymptotic normality of $\hat \beta$ in the case of the copula families in Subsection~\ref{subs_examples}.
\mds

{\bf\it Example 4 cont'd: the Gaussian copula.}

Obviously, Assumptions~\ref{dini},~\ref{hessc} and~\ref{lipschitz} are satisfied. This is the case for Assumption~\ref{d0} too, because $\Sigma \mapsto \ln (|\Sigma|)$ is Lipschitz under~(\ref{theta_max}) and invoking Example 19.7 in van der Vaart (2007).

\mds

To deal with Assumption~\ref{d6}, note that $p$ and $\hat p$ are Lipschitz transforms of conditional
Kendall's tau $\tau(\beta,\beta'\zb)$ and $\hat \tau(\beta,\beta'\zb)$ respectively. Due to Example 19.20 in van der Vaart (2007), it is sufficient to show that
the functions $\zb \mapsto \nabla_\beta\tau (\beta_0,\beta_0'\zb)$
and $\zb \mapsto \nabla_\beta\hat\tau (\beta_0,\beta_0'\zb)$ belong to a Donsker class a.e., assuming the underlying dimension $d$ is two.
It follows from Lemma \ref{tricky1} and from the relation $\tau(\beta_0,\beta_0'\zb)=4\int C_{\beta_0}(\mathbf{u}|\beta_0'\zb)C_{\beta_0}(d\mathbf{u}|\beta_0'\zb)-1$ that
$$\nabla_\beta \tau (\beta_0,\beta_0'\zb)=f_1(\beta_0'\zb)+\zb f_2(\beta_0'\zb),\; \zb\in \Zc,$$
with
\begin{eqnarray*}
f_1(v) &=&  -E[\ZZ|\beta_0'\ZZ=v,\ZZ\in \mathcal{Z}]\left\{\int c_0(\mathbf{u},v)C_{\beta_0}(d\mathbf{u}|v)+\int C_{\beta_0}(\mathbf{u}|v)c_0(d\mathbf{u},v)\right\},\\
f_2(v) &=& \ZZ \left\{\int c_0(\mathbf{u},v)C_{\beta_0}(d\mathbf{u}|v)+\int C_{\beta_0}(\mathbf{u}|v)c_0(d\mathbf{u},v)\right\},
\end{eqnarray*}
using the notations of Lemma \ref{tricky1}.
In a Gaussian copula family, $\zb\mapsto f_{j}(\beta_0'\zb)$ and $\zb\mapsto f'_j(\beta_0'\zb),$
 are uniformly bounded on $\Zc$. Therefore, $\nabla_\beta \tau (\beta_0,\beta_0'\zb)$ belongs to the class $\mathcal{G}=\{\zb\in \Zc\rightarrow f(\beta_0'\zb)+\zb g(\beta_0'\zb), f,g\in \mathcal{C}(M)\},$ with $\mathcal{C}(M)=\{f:\|f\|_{\infty}+\|f'\|_{\infty}\leq M\}.$ $\mathcal{C}(M)$ is a Donsker class from Theorem 2.7.1 in Van der Vaart and Wellner (1996). Moreover, $\mathcal{G}$ is Donsker from Examples 2.10.7 and 2.10.8 in Van der Vaart and Wellner (1996).

\mds

It is also the case for $\nabla_\beta\hat\tau $.
Indeed, with the notations of Section~\ref{condit_tau}, we can write
$$ \hat\tau (\beta,\beta'\zb)=\frac{4}{n^2 \hat f^2_{\beta}(\beta'\zb)} \sum_{i,j=1}^n \1(\XX_j\leq \XX_i) \tilde K_{\tilde h}\left(\beta'\ZZ_j -\beta'\zb \right)
\tilde K_{\tilde h}\left(\beta'\ZZ_i -\beta'\zb \right)-1. $$
A differentiation with respect to $\beta$ easily shows that $\nabla_\beta\hat{\tau} (\beta_0,\beta_0'\zb)$ is of the form
$$\nabla_\beta\hat{\tau} (\beta_0,\beta_0'\zb)=\hat{f}_1(\beta_0'\zb)+\zb \hat{f}_2(\beta_0'\zb).$$
The results of Section \ref{condit_tau} allow to show that $\sup_{\zb\in \Zc}|\hat{f}_j(\beta_0'\zb)-f_j(\beta_0'\zb)|=O_P(\tilde h_n^2+[\log n]^{1/2}n^{-1/2}\tilde h_n^{-3/2}),$ and that $\sup_{\zb\in\Zc}|\hat{f}'_j(\beta_0'\zb)-f'_j(\beta_0'\zb)|=O_P(\tilde h_n^2+[\log n]^{1/2}n^{-1/2}\tilde h_{n}^{-5/2}),$ for $j=1,2.$ Therefore,
$\zb\in \Zc\mapsto \nabla_\beta\hat{\tau} (\beta_0,\beta_0'\zb)$ belongs to the Donsker class $\mathcal{G}$ when $n\tilde h_n^5 \rightarrow 0$.

\mds

Assumption~\ref{hess} is coming from the results of Section~\ref{condit_tau}, and simple calculations prove that Assumption~\ref{uchapeau} is satisfied for every $\alpha >0$.

\mds

Recalling the notations of Section~\ref{condit_tau}, we have
$$\sup_{\zb\in \Zc}|\hat{\tau}(\beta_0,\beta_0'\zb)-\tau(\beta_0,\beta_0'\zb)|=O_P(\tilde h^{\tilde s}+[\log n]^{1/2}n^{-1/2}\tilde{h}^{-1/2}):=O_P(\eta_{1n}),\;\text{and}$$
$$\sup_{\zb\in \Zc}|\nabla_{\beta}\hat{\tau}(\beta,\beta_0'\zb)-\nabla_{\beta_0}\tau(\beta_0,\beta_0'\zb)|=O_P(\tilde h^{\tilde s}+[\log n]^{1/2}n^{-1/2}\tilde h^{-3/2})
:=O_P(\eta_{2n}).$$
To fix the ideas, assume $\tilde h \sim n^{-\tilde\pi}$, for some $\tilde\pi>0$. Then, to satisfy $\eta_{1n}\eta_{2n}=o(n^{-1/2})$, it is sufficient we have
$ 4\tilde s \tilde \pi >1$, $\tilde s \geq 2$ and $4\tilde\pi <1$.
Recall that we had set $\delta_n\sim n^{-\pi s}+\ln_2 n .n^{-(1-p \pi)/2}$. To satisfy $\delta_n^{1-\alpha}\eta_{jn}=o(n^{-1/2})$, $j=1,2$, it is sufficient to have
$$ 1< (1-\alpha)\min(2s\pi,1-p\pi) + \min( 2\tilde s \tilde \pi,1-3\tilde \pi).$$

\mds

Concerning Assumption~\ref{ToManageTrimming}, it can be verified that the l.h.s. of~(\ref{eq_ToManageTrimming}) is $O\left( \delta_n \nu_n[\Phi^{-1}(\nu_n)]^{2}   \right)$.
Nonetheless, $\Phi^{-1}(\nu_n)\sim - \sqrt{  (-2) \ln \nu_n}$, when $\nu_n \rightarrow 0$ (see Dominici, 2003). A sufficient condition is then $\delta_n \nu_n |\ln(\nu_n)| = o(n^{-1/2})$.

Assumptions~\ref{my_hatF} and~\ref{bias} are trivially satisfied because we have chosen nonparametric marginal cdfs' and we apply Lemma~\ref{my_Fk}, for which we have seen that we set $W(\zb,\xx)=0$.

Assumption~\ref{deriv_u2} is the most demanding and cannot be obtained by the same reasoning as for Assumption~\ref{ToManageTrimming}.
Actually, we recall that the former one has been requested only in the proof of Corollary~\ref{ouf} to show that
$$ \frac{1}{n}\sum_{i=1}^n
\nabla_{\uu}\nabla^2_{\theta} (\ln c_{\theta})_{|\theta = \psi_i}(\UU_i^*)
 .[\hat{\UU}_i-\UU_i]^2\nabla_{\beta}\psi(\beta,\beta'\ZZ_i)_{|\beta=\beta_0}\omega_{i,n} = o_P(n^{-1/2}),$$
for some random vectors $\UU_i^*$, $|\UU_i^* -\UU_i|\leq |\hat\UU_i -\UU_i|$.
Due to Assumption~\ref{A_unifapprox}, it is sufficient to check that
$$ \delta_n^2 E\left[ | \nabla_{\uu}\nabla^2_{\theta} (\ln c_{\theta})_{|\theta = \psi_i}(\UU_i)
\nabla_{\beta}\psi(\beta,\beta'\ZZ_i)_{|\beta=\beta_0}|\omega_{i,n}\right] = o(n^{-1/2}).$$
Due to the bounded-ness of $c_\theta$, the latter expectation is less than a constant times
$$ \int_{\Phi^{-1}(\nu_n)}^{\Phi^{-1}(1-\nu_n)} |t| \exp(t^2/2)\, dt.$$
The latter integral behaves as $ \exp\left( [\Phi^{-1}(\nu_n)]^2/2 \right)$.
Since  $\Phi^{-1}(\nu_n)\sim - \sqrt{  (-2) \ln \nu_n}$, it is sufficient to satisfy
$ \delta_n^2 /\nu_n = o(n^{-1/2}).$
Usual variance calculations for kernel densities prove that Assumption~\ref{uprocess2} is true when $nh^p=n^{1-p\pi} \rightarrow \infty$, i.e. when $p\pi <1$.

\mds

Gathering all the previous constraints, we can exhibit explicit combinations of parameters. For instance, we can set
$$s=2p ,\; \tilde s=4 ,\;  \pi = 1/(2s+p),\; \tilde \pi = 1/9,
h_n\sim n^{-1/(2s+p)}=n^{-1/5p},\;\tilde h_n\sim n^{-4/9}, \alpha <1/2,$$
implying $ \delta_n \sim n^{-2/5}$ and we choose $ \nu_n = n^{-1/5}$.
Note that we need high-order kernels in general, even in the bivariate case ($p=2$).

\mds

Similar reasonings allow to exhibit explicit tuning parameters to manage Clayton and/or Gumbel copula models. They are left to the reader as an exercise.

\section{Conditional Kendall's Tau}
\label{condit_tau}
In this section, we show how to check Assumptions~\ref{hess} and~\ref{dbis} in general, when the conditional margins are estimated non-parametrically.
Incidentally, we prove some theoretical results related to the estimation of conditional Kendall's tau, that are valuable per se.

\mds

We consider the situation of a $d$-dimensional random vector $\XX,$ whose conditional copula is parameterized by $\tau(\beta,\beta'\zb)$, the conditional Kendall's tau coefficient of this vector as defined in~(\ref{tauHbeta}) when $d=2$, and~(\ref{defTau_d}) more generally.
In other words, we consider the case where $\psi(\beta,\beta'\zb)=\Phi(\tau(\beta,\beta'\zb))$ for some ``sufficiently regular'' function $\Phi$.
Indeed, Kendall's tau are commonly used for inference purpose of parametric copulae, particularly Archimedean and elliptical copulae.
Moreover, as explained in Subsection~\ref{psi_mgt}, (A1) and (A2) are satisfied in such cases. Finally, we do not suffer from the curse of dimension because conditional Kendall's tau
are those associated to the copula of $\XX$ knowing $\beta'\ZZ$.

\mds

Introducing a kernel estimator $\hat{F}_{\beta}$ of $F_{\beta}(\xx|y)=\mathbb{P}(\XX\leq \xx|\beta'\ZZ=y)$ as
$\hat{F}_{\beta}(\xx|y)=\hat{H}_{\beta}(\xx,\infty|y)$ (recall~(\ref{def_Hbeta})), define
\beq \nonumber
\hat{\tau}(\beta,\beta'\zb) = \frac{1}{2^d -1}\left\{ 2^d\int \hat{F}_\beta(\xx | \beta'\zb) \hat{F}_\beta(d\xx | \beta'\zb)-1\right\}.
\eeq

In Lemma \ref{con_tau} below, we show that the uniform consistency of the conditional Kendall's tau coefficient is obtained, provided that we have some convenient convergence rates for $\hat{F}_{\beta}.$

\begin{lemma}
\label{con_tau}
Assume that
\begin{eqnarray}
\sup_{\xx\in \Rb^d,\beta \in \Bc,\zb \in \Zc} |\hat{F}_{\beta}(\xx|\beta'\zb)-F_{\beta}(\xx|\beta'\zb)| &=& O_P(\varepsilon_{n,0})\label{unif_const1}. \label{unif_const2}
\end{eqnarray}
Then,
$$\sup_{\beta \in \Bc,\zb\in \Zc}|\hat{\tau}(\beta,\beta'\zb)-\tau(\beta,\beta'\zb)|=O_P(\varepsilon_{n,0}).$$
\end{lemma}
\begin{proof}
Decompose
\begin{eqnarray*}
(2^d -1)\left\{\hat{\tau}(\beta,\beta'\zb)-\tau(\beta,\beta'\zb)\right\} &=& 2^d\int \{\hat{F}_\beta(\xx | \beta'\zb)-F_\beta(\xx | \beta'\zb) \}\hat{F}_\beta(d\xx | \beta'\zb)\\
&+ & 2^d\int F_\beta(\xx | \beta'\zb)\{\hat{F}_\beta(d\xx | \beta'\zb)-F_\beta(d\xx | \beta'\zb)\}.
\end{eqnarray*}
The first term is $O_P(\varepsilon_n)$ due to (\ref{unif_const1}). For the second, observe that
$$\int F_\beta(\xx | \beta'\zb) \{ \hat F_\beta(d\xx | \beta'\zb) - F_\beta(d\xx | \beta'\zb)\} = (-1)^{d-1}
\int  \{ \hat F_\beta(\xx | \beta'\zb) - F_\beta(\xx | \beta'\zb)\}\, F(d\xx | \beta'\zb),$$
which is less than $\sup_{\xx,\beta,\zb} |\hat{F}_{\beta}(\xx|\beta'\zb)-F_{\beta}(\xx|\beta'\zb)|,$ and we use again (\ref{unif_const1}).
\end{proof}

\mds

Lemma \ref{con_tau} provides some tools to verify the first part of Assumptions \ref{hess} and \ref{dbis}, if one assumes that the function $\Phi$ is regular enough (that is H\"older with some high enough H\"older exponent).
Similarly, we can derive the uniform consistency of $\nabla_{\beta}^j \hat{\tau}$ for $j=1,2,$ which allows to check the remaining conditions in Assumptions \ref{hess} and \ref{dbis}.

\begin{lemma}
\label{con_tau2}
Assume that
\begin{eqnarray}
\sup_{\xx\in \Rb^d,\beta\in \Bc,\zb\in \Zc} |\nabla_{\beta}^j\hat{F}_{\beta}(\xx|\beta'\zb)-\nabla_{\beta}^jF_{\beta}(\xx|\beta'\zb)| &=& O_P(\varepsilon_{n,j})\label{unif_const4},
\end{eqnarray}
for $j=1,2$, and that
$$\sup_{j=1,2}\int \left|\nabla^j_{\beta}F_{\beta}(d\xx|\beta'\zb)\right| + \left|\nabla^j_{\beta}\hat F_{\beta}(d\xx|\beta'\zb)\right|\leq M,$$ for some $M>0.$
Then,
$$\sup_{\beta\in \Bc,\zb\in \Zc}|\nabla_{\beta}\hat{\tau}(\beta,\beta'\zb)-\nabla_{\beta}\tau(\beta,\beta'\zb)|=O_P(\max(\varepsilon_{n,1}, \varepsilon_{n,0})),\;\text{and}$$
$$\sup_{\beta\in \Bc,\zb \in \Zc}|\nabla_{\beta}^2 \hat{\tau}(\beta,\beta'\zb)-\nabla_{\beta}^2 \tau(\beta,\beta'\zb)|=O_P (\max(\varepsilon_{n,2},\varepsilon_{n,1},\varepsilon_{n,0})).$$
\end{lemma}
\begin{proof}
This is a consequence of applying the $\nabla-$operator to
$\hat{\tau}(\beta,\beta'\zb)$, and of the compactness of $\Zc$.
\end{proof}

The next step is to verify that, under reasonable conditions, (\ref{unif_const1}) and (\ref{unif_const4}) hold. To this aim, let us introduce some assumptions.

\begin{assum}
\label{akernel}
Let $\tilde K$ denote a univariate symmetric kernel function of order $\tilde s$, $s\geq 2$. It is twice continuously differentiable with bounded derivatives up to order $2$.
Moreover, $(\tilde h_n)$ denotes a bandwidth sequence, where $\tilde h_n=O(n^{-a})$ for some $a>0$ and $n\tilde h_n \rightarrow \infty$.
\end{assum}
Note that, in general, the latter triplet $(\tilde K,\tilde h,\tilde s)$ is different from the similar quantities $(K,h,s)$ that have been invoked to define the pseudo-observations $\hat\UU_i$ (see Lemma~\ref{my_Fk}).
\begin{assum}
\label{abounded} Let $f_{\beta}(y)$ denote the density of $\beta'\ZZ$ evaluated at point $y$.
Assume that $\inf_{\beta\in \Bc,\zb\in \Zc}\inf_y f_{\beta}(y)>c,$ for some $c>0.$ Moreover, assume that $f_{\beta}$ is $s$-times continuously differentiable,
with uniformly bounded derivatives.
\end{assum}

The latter assumption is satisfied most of the time, because $\beta'\zb$ belongs to a compact subset when $\beta\in\Bc$ and $\zb\in\Zc$~\footnote{For instance, assume the arguments $y$ above belong to a fixed interval $[a,b]$ and that $\ZZ$ follows a Gaussian $\Nc(0,\Sigma)$. Then $\beta'\ZZ\sim \Nc(0, \beta'\Sigma\beta)$ and
$f_\beta(y)=\exp(-y^2/2(\beta'\Sigma \beta))/(\sqrt{2\pi} \beta'\Sigma \beta).$ Since $\beta'\Sigma \beta$ belongs to a compact $[c,d]$, $c>0$, the latter density is
larger than $\exp(-b^2/(2d^2))/(\sqrt{2\pi} d) >0$.}.
In the single-index literature, some authors relaxed this assumption, by only assuming
$\inf_{\zb}\inf_y f_{\beta_0}(y)>c.$ Nevertheless, Assumption \ref{abounded} requires to introduce a trimming procedure, in order to avoid parts of the distribution for which some $f_{\beta}(\beta'\ZZ_i)$ are too close to zero. Such trimming procedures (generally working in two-steps), that can be extended straightforwardly in our case, have been investigated in detail for example in Lopez, Patilea, Van Keilegom (2013).

\mds

Let $\mathcal{A}$ denote a generic set of functions with envelope $F$. For a probability measure $\mathbb{Q},$ let $\mathcal{N}(\varepsilon,\mathcal{A},\|\cdot\|_{2,\mathbb{Q}})$ denote the number of $L^2(\mathbb{Q})-$balls required to cover the set of functions $\mathcal{A},$ and $N(\varepsilon,\mathcal{A})=\sup_{\mathbb{Q}:\|F\|_{2,\mathbb{Q}}<\infty}\mathcal{N}(\varepsilon\|F\|_{2,\mathbb{Q}},\mathcal{A},\|\cdot\|_{2,\mathbb{Q}}).$

\begin{assum}
\label{aclass}
$\mathcal{A}$ is a class of functions bounded by $1$ such that
$N(\varepsilon,\mathcal{A})\leq C\varepsilon^{-\nu}.$ Moreover, for $\phi \in \mathcal{A}$,
let $m_{\phi}(y)=E[\phi(\XX,\ZZ)|\beta'\ZZ=y]$. Assume that the functions $m_{\phi}$ are twice continuously differentiable, and their derivatives
up to order $2$ are bounded by some finite constant $M$ that does not depend on $\phi$.
\end{assum}

We first state Lemma \ref{e_mason} that provides consistency rates for kernel weighted sums.

\begin{lemma}
\label{e_mason}
Let $\mathcal{L}$ denote a class of functions satisfying Assumption \ref{aclass}.
Under Assumption \ref{akernel}, we have
\bqa
\lefteqn{ \frac{1}{n\tilde h}\sup_{\lambda\in \Lc}\sup_{\beta\in \Bc,\zb\in \Zc}\left|\sum_{i=1}^n \lambda(\XX_i,\ZZ_i)\tilde K\left(\frac{\beta'\ZZ_i-\beta'\zb}{\tilde h}\right)-E\left[\lambda(\XX_i,\ZZ_i)\tilde K\left(\frac{\beta'\ZZ_i-\beta'\zb}{\tilde h}\right)\right]\right|  }\\
&=& O_P([\log n]^{1/2}n^{-1/2}\tilde h^{-1/2}).\hspace{7cm}
\eqa
\end{lemma}

\begin{proof}
Let $$\mathcal{B}=\sup_{\beta,\zb,\lambda}\left|\sum_{i=1}^n \lambda(\XX_i,\ZZ_i)\tilde K\left(\frac{\beta'\ZZ_i-\beta'\zb}{\tilde h}\right)-E\left[\lambda(\XX_i,\ZZ_i)\tilde K\left(\frac{\beta'\ZZ_i-\beta'\zb}{\tilde h}\right)\right] \right|,$$
and $$\mathcal{B}_{\varepsilon}=E\left[\sup_{\beta,\zb,\lambda}\left|\sum_{i=1}^n \varepsilon_i\lambda(\XX_i,\ZZ_i)\tilde K\left(\frac{\beta'\ZZ_i-\beta'\zb}{\tilde h}\right)\right|\right],$$
where $(\varepsilon_i)_{1\leq i \leq n}$ are i.i.d. Rademacher variables.
Due to Proposition \ref{prop_talagrand}, we have
\begin{equation}\label{ineg_expo}
\mathbb{P}\left(\mathcal{B} \geq A_1(\mathcal{B}_{\varepsilon}+t)\right)\leq 2\left\{\exp(-A'_2t^2/(n\tilde h))+\exp(-A_2t)\right\},
\end{equation}
where $A'_2$ is a constant. Indeed, since the function $\lambda$ are uniformly and bounded by one,
$$\sup_{\beta,\zb,\lambda} Var \left(\lambda(\XX,\ZZ)\tilde K\left(\frac{\beta'\ZZ-\beta'\zb}{\tilde h}\right)\right)=O(\tilde h).$$
Next, observe that the class of functions
$$\mathcal{L}_{\tilde K}=\left\{g: \Rb^d\times \Zc  \rightarrow \Rb, (\xx,\zb)\mapsto \lambda(\xx,\zb)\tilde K\left(\frac{\beta'\zb-\beta'\textbf{u}}{\tilde h}\right):\textbf{u}\in \Zc,\beta\in \Bc, \tilde h\in \Rb^+\right\},$$
satisfies the assumptions of Proposition \ref{prop_rade} with $\sigma^2=O(h)$ and
\begin{equation}
N(\varepsilon,\mathcal{L}_{\tilde K})\leq C\varepsilon^{-\nu}, \label{covering}
\end{equation}
for some $C$ and $\nu.$ The property (\ref{covering}) can be obtained from the following: Lemma 22 in Nolan and Pollard (1987) shows that
$N(\varepsilon,\mathcal{K})\leq C_2\varepsilon^{-\nu_2},$ where
$$\mathcal{K}=\left\{(\xx,\zb)\in \Rb^d\times \Zc\mapsto \tilde K\left(\frac{\beta'\zb-\beta'\textbf{u}}{\tilde h}\right):\textbf{u}\in \Zc,\beta\in \Bc, \tilde h\in \Rb^+\right\}.$$
Using Assumption \ref{aclass} and Lemma A.1 in Einmahl and Mason (2000), we get that $\mathcal{L}_{\tilde K}= \mathcal{L}\cdot \mathcal{ K}$ satisfies (\ref{covering}).

\mds

Therefore, we can apply Proposition \ref{prop_rade} to deduce that
\begin{equation}\mathcal{B}_{\varepsilon}\leq A'n^{1/2}\tilde h^{1/2}[\log (\tilde h^{-1})]^{1/2}=A''n^{1/2}\tilde h^{1/2}[\log n]^{1/2}. \label{BE}
\end{equation}
It follows from (\ref{BE}) that, for $t_1>2A_1A'',$
$$\mathbb{P}(\mathcal{B}\geq t_1n^{1/2}\tilde h^{1/2}[\log n]^{1/2})\leq \mathbb{P}\left(\mathcal{B}\geq A_1\mathcal{B}_{\varepsilon}+t_1n^{1/2}\tilde h^{1/2}[\log n]^{1/2}/2\right).$$
Applying (\ref{ineg_expo}) with $t=t_1 n^{1/2}\tilde h^{1/2}[\log n]^{1/2}/(2A_1)$, we get
$$\mathbb{P}(\mathcal{B}\geq t_1 n^{1/2}\tilde h^{1/2}[\log n]^{1/2})\leq 2\left\{\exp(-A_2' t_1^2[\log n]/(4A_1^2))+\exp(-A_2t_1n^{1/2}\tilde h^{1/2}[\log n]^{1/2}/(2A_1))\right\},$$
and the result follows.
\end{proof}

This Lemma is the cornerstone of Lemma \ref{consist_F} belows, which ensures consistency rates for $\hat{F}_{\beta}$ and its derivatives.

\begin{lemma}
\label{consist_F}
Let $\mathcal{A}$ denote a class of functions satisfying Assumption \ref{aclass}.
Then, under Assumptions \ref{akernel} and \ref{abounded},
\bqa
\lefteqn{
\sup_{\phi\in \mathcal{A}}\sup_{\beta\in \Bc,\zb\in\Zc}\left|\int \phi(\xx,\zb)\{\hat{F}_\beta(d\xx | \beta'\zb)-F_\beta(d\xx | \beta'\zb)\}\right| }\\
&=&O_P\left(\tilde h^{\tilde s}+[\log n]^{1/2}n^{-1/2}\tilde h^{-1/2}\right).
\eqa
\end{lemma}

\begin{proof}
Write
$$\hat m_{\phi}(\beta'\zb):=\int \phi(\xx,\zb)\hat{F}_\beta(d\xx | \beta'\zb)=\frac{1}{n\tilde h\hat{f}_{\beta}(\beta'\zb)}\sum_{i=1}^n \phi(\XX_i,\ZZ_i)\tilde K\left(\frac{\beta'\ZZ_i-\beta'\zb}{\tilde h}\right),$$
where
\begin{equation}
\hat{f}_{\beta}(\beta'\zb)=\frac{1}{n\tilde h}\sum_{i=1}^n \tilde K\left(\frac{\beta'\ZZ_i-\beta'\zb}{\tilde h}\right),
\label{def_kernel_density}
\end{equation}
is an estimator of the density $f_{\beta}(\beta'\zb)$ of $\beta'\ZZ$ evaluated at $\beta'\zb.$
Let $$\hat{\frak{m}}_{\phi}(\beta'\zb)=\frac{1}{n\tilde h}\sum_{i=1}^n \phi(\XX_i,\ZZ_i)\tilde K\left(\frac{\beta'\ZZ_i-\beta'\zb}{\tilde h}\right)=
\hat m_{\phi}(\beta'\zb)\hat f_{\beta}(\beta'\zb),$$
and $\frak{m}_{\phi}(\beta'\zb)=m_{\phi}(\beta'\zb)f_{\beta}(\beta'\zb).$
It follows from Lemma \ref{e_mason} that
$$\sup_{\beta,\zb,\phi}|\hat{\frak{m}}_{\phi}(\beta'\zb)-E[\hat{\frak{m}}_{\phi}(\beta'\zb)]|+\sup_{\beta,\zb}|\hat{f}_{\beta}(\beta'\zb)-E[\hat{f}_{\beta}(\beta'\zb)]|
=O_P(\frac{[\log n]^{1/2}}{n^{1/2}\tilde h^{1/2}}).$$
Moreover, using classical arguments on kernel estimators (and Assumptions \ref{aclass} and \ref{akernel}), we have
$$\sup_{\beta,\zb,\phi}|E[\hat{\frak{m}}_{\phi}(\beta'\zb)]-\frak{m}_{\phi}(\beta'\zb)|+\sup_{\beta,\zb}|E[\hat{f}_{\beta}(\beta'\zb)]-f_{\beta}(\beta'\zb)|=O(\tilde h^{\tilde s}).$$
The result of the Lemma follows from the fact that the density $f_{\beta}(\beta'\zb)$ is bounded away from zero by Assumption \ref{abounded}.
\end{proof}

Lemma \ref{consist_F} allows to verify condition (\ref{unif_const1}) by considering $\phi(\xx,\zb)=\1(\xx\leq \xx_0)$, for some constant vectors $\xx_0$. This shows that, in this case,
$\varepsilon_{n,0}=\tilde h^{\tilde s}+[\log n]^{1/2}n^{-1/2}\tilde h^{-1/2}.$ It also permits to obtain the uniform consistency rates for $\nabla_{\beta}^j \hat{F}_{\beta}$ for $j=1,2,$ with
$$ \varepsilon_{n,1} = \tilde h^{\tilde s} +\frac{[\log n]^{1/2}}{n\tilde h^{3/2}}, \;\varepsilon_{n,2} = \tilde h^{\tilde s} +\frac{[\log n]^{1/2}}{n\tilde h^{5/2}}.$$
Indeed,
$$\nabla_{\beta} \hat{\frak{m}}_{\phi}(\beta'\zb)=\frac{1}{n\tilde h^2}\sum_{i=1}^n \1(\XX_i\leq \xx).(\ZZ_i-\zb) \tilde K'\left(\frac{\beta'\ZZ_i-\beta'\zb}{\tilde h}\right),$$ and the convergence of this term can be studied using Lemma \ref{consist_F}, but replacing $\tilde K$ by $\tilde K',$ and setting $\phi(\XX,\ZZ)=\1(\XX \leq \xx).(\ZZ-\zb)$. The latter function is indexed by $(\xx,\zb)$ that
lives into $\Rb^d \times \Zc$, defining the convenient class $\Ac$ to apply Lemma~\ref{consist_F}.
The other terms obtained by differentiation can be studied in the same way.

\mds

Hence, the latter results allow to verify whether Assumptions~\ref{hess} and~\ref{dbis} hold. Indeed, under some (light) conditions of regularity, we have obtained that
$$\sup_{\beta \in \Bc,\zb\in \Zc}|\hat{\tau}(\beta,\beta'\zb)-\tau(\beta,\beta'\zb)|=O_P(\tilde h^{\tilde s}+[\log n]^{1/2}n^{-1/2}\tilde{h}^{-1/2}),$$
$$\sup_{\beta\in \Bc,\zb\in \Zc}|\nabla_{\beta}\hat{\tau}(\beta,\beta'\zb)-\nabla_{\beta}\tau(\beta,\beta'\zb)|=O_P(\tilde h^{\tilde s}+[\log n]^{1/2}n^{-1/2}\tilde h^{-3/2}),\;\text{and}$$
$$\sup_{\beta\in \Bc,\zb \in \Zc}|\nabla_{\beta}^2 \hat{\tau}(\beta,\beta'\zb)-\nabla_{\beta}^2 \tau(\beta,\beta'\zb)|=O_P (\tilde h^{\tilde s}+[\log n]^{1/2}n^{-1/2}\tilde h^{-5/2}).$$

\mds

\section{A short simulation study}

To investigate the finite sample behavior of our procedure, we consider the following simulation setting. Let $\ZZ\in \mathbb{R}^4$ be a random vector with independent margins that are uniformly distributed over $[0,1].$ The conditional (marginal) distributions of $X_1$ and $X_2$ given $\ZZ=\zb$ are respectively taken as
\begin{equation*}
X_1|\ZZ=\zb \sim  \mathcal{N}(\alpha'\zb,1), \;\;X_2|\ZZ=\zb \sim  \mathcal{N}(\alpha'\zb,1),
\end{equation*}
with $\alpha=(0.2,-0.3,0.1,0.1)'.$
Concerning the conditional copula model, we consider the case where the dependence structure of $(X_1,X_2)$ given $\ZZ=\zb$ is a Clayton copula of parameter $\beta_0'\zb.$
We recall that a Clayton copula with a parameter $\theta >0$ is defined as $C_{\theta}(u,v)=(u^{-\theta}+v^{-\theta}-1)^{-1/\theta}.$
We set $\beta_0=(\beta_{0,1},\beta_{0,2},\beta_{0,3},\beta_{0,4})=(1,0.7,1.2,0.8)'.$ We then assume
$$\theta(\zb)=\psi(\beta_0,\beta_0'\zb)=\frac{\beta_0'\zb/3.7}{1-\beta_0'\zb/7.4}\cdot$$
Therefore, the values of the random parameter $\theta(\ZZ)$ belong to $(0,2)$.
The sample size is $n=1000.$ We consider $1000$ replications of this simulation scheme. For each simulated sample, we compute our estimator $\hat{\beta}$ as defined in~(\ref{true_criterion}). The conditional margins are estimated using a kernel estimator as in (\ref{KernelFk}), with products of univariate Gaussian kernels $K_k=\Phi$, $k=1,\ldots,4$.
The bandwidth used in the estimation of the margins is taken as $h=0.12$ (Scott's rule in $\Rb^4$: see Scott 2015, p.164). On the other hand,
we consider different values of the bandwidth for computing the conditional Kendall's $\tau$ estimators, to understand its impact on the performances. To measure the estimation errors, for each simulated sample and each value of the bandwidth, we introduce
$$e(\hat{\beta})=\frac{1}{3}\sum_{i=2}^4\frac{|\hat{\beta}_i-\beta_{0,i}|}{|\beta_{0,i}|}\cdot$$
The latter number measures the mean relative absolute distance between the true parameter $\beta_0$ and its estimate.
The mean value of $e(\hat{\beta})$ over the 1000 replications is then calculated, providing $\bar e(\hat{\beta})$.
For a grid of realistic bandwidths between $0.03$ and $0.2$, the obtained range of $\bar e(\hat{\beta})$ is $[1.3\%, 6.1\%]$, with an average value of $3.5\%$.
Such error levels that are very reasonable.
As expected, there is an impact of the bandwidth on the practical behavior. However, we focus here on estimating $\beta_0,$ which is relatively weakly sensitive to the choice of the bandwidth. Indeed, it is based on the maximization of a criterion where the nonparametric kernel estimators are averaged.

\mds

\begin{center}
{\bf Acknowledgements}
\end{center}
\noindent
The authors thank Axel B\"ucher, Holger Dette, Christian Genest, Johan Segers, Alexandre Tsybakov, Stanislav Volgushev for helpful discussions, and
numerous seminar participants, particularly at the 2015 European Meeting of Statisticians in Amsterdam, and at the SFB 2015 workshop in Bochum.
This research has been supported by the Labex Ecodec.

\mds

\mds

Jean-David Fermanian, Ensae-Crest, J120, 3 av. Pierre Larousse, 92245 Malakoff cedex. France.
E-mail: jean-david.fermanian@ensae.fr

\appendix

\section{Technical lemmas}

\begin{lemma}
\label{tricky2}
Consider an integrable function $\chi$ on $(0,1)^d \times \Zc$.
Assume that there exist two deterministic sequences $(\xi_n)$ and $(\delta_n)$, $\xi_n\rightarrow 0$, $\delta_n = o(\nu_n)$, s.t.
$ \Pb \left( \sup_i|\hat\UU_i - \UU_i|> 2 \delta_n ,\ZZ_i\in\Zc\right)\longrightarrow 0$ when $n\rightarrow \infty$, and
\beq
E\left[ |\chi(\UU_i,\ZZ_i)| \cdot \1(\ZZ_i\in\Zc)\{ \1(|U_{i,k} -\nu_n | \leq  2\delta_n)+ \1(|1 -\nu_n -U_{i,k} | \leq  2\delta_n) \} \right]
\leq \xi_n ,
\label{condBoundary}
\eeq
for all $k=1,\ldots,d$. Then
$n^{-1}\sum_{i=1}^n |\chi(\UU_i,\ZZ_i).(\hat\omega_{i,n}-\omega_{i,n})|=O_P(\xi_n)$.
\end{lemma}

\begin{proof}
Let us fix $\ep>0$. For any constant $A>0$, we have
\bqa
\lefteqn{    \Pb \left( \frac{1}{n}\sum_{i=1}^n |\chi(\UU_i,\ZZ_i)| \cdot |\hat\omega_{i,n}-\omega_{i,n}| > A \xi_n \right) } \\
&\leq &
   \Pb \left( \frac{1}{n}\sum_{i=1}^n |\chi(\UU_i,\ZZ_i)| \cdot |\hat\omega_{i,n}-\omega_{i,n}|\cdot \1(|\hat\UU_i - \UU_i| \leq 2 \delta_n) > A\xi_n \right)  \\
 & + & \Pb \left( \sup_i|\hat\UU_i - \UU_i|> 2 \delta_n,\ZZ_i\in\Zc \right) :=\Pb_1 + \Pb_2.
\eqa
First, we have
\bqa
\lefteqn{ \Pb_1 \leq
\sum_{k=1}^d \Pb \left( \frac{1}{n}\sum_{i=1}^n |\chi(\UU_i,\ZZ_i)| \cdot \1(\ZZ_i\in\Zc, |U_{i,k} -\nu_n | \leq  |\hat U_{i,k} -U_{i,k} |
 ) > A \xi_n /(2d) \right) }\\
&+ &
\sum_{k=1}^d \Pb \left( \frac{1}{n}\sum_{i=1}^n |\chi(\UU_i,\ZZ_i)| \cdot \1(\ZZ_i\in\Zc, |1 -\nu_n -U_{i,k}| \leq  |\hat U_{i,k} -U_{i,k} |
 ) >A\xi_n /(2d) \right) \\
&\leq &  \frac{2d}{A \xi_n} \sum_{k=1}^d E\left[
|\chi(\UU_i,\ZZ_i)| \cdot \1(\ZZ_i\in\Zc).\{\1 (|U_{i,k} -\nu_n | \leq  2\delta_n)+ \1 (|1 -\nu_n -U_{i,k} | \leq  2\delta_n) \} \right] \\
&\leq & 2d /A,
\eqa
that is less than $\ep>0$ for $A$ large enough. This means $\Pb_1=O_P(\xi_n)$. Second, by assumption, $\Pb_2$ is less than $\ep $ when $n$ is sufficiently large, proving the result.
\end{proof}

\bre
In particular, it is tempting to define, with obvious notations,
\bqa
\lefteqn{ \xi_n :=\sup_k E[\sup_{u_k, |u_k - \nu_n|\leq 2\delta_n}
|\chi(u_k,\UU_{i,-k},\ZZ_i)|\cdot \1(\ZZ_i\in\Zc)] }\\
&+& \sup_k E[\sup_{u_k, |u_k - 1 + \nu_n|\leq 2\delta_n}
|\chi(u_k,\UU_{i,-k},\ZZ_i)|\cdot \1(\ZZ_i\in\Zc)],
\eqa
or even, when it tends to zero,
\bqa
\lefteqn{  \xi_n:= \sup_k\sup_{u_k, |u_k - \nu_n|\leq 2\delta_n}\sup_{\uu_{-k} \in [\nu_n-2\delta_n,1-\nu_n+2\delta_n]^{d-1}}  \sup_{\zb\in\Zc}
|\chi(u_k,\uu_{-k},\zb)| }\\
&+&
\sup_k\sup_{u_k, |u_k - 1+\nu_n|\leq 2\delta_n}\sup_{\uu_{-k} \in [\nu_n-2\delta_n,1-\nu_n+2\delta_n]^{d-1}}  \sup_{\zb\in\Zc}
|\chi(u_k,\uu_{-k},\zb)|.
\eqa

\ere

\begin{lemma}
\label{lemmahessian}
Under the assumptions of Theorem \ref{asr},
$$\sup_{\beta\in \Bc}|\nabla^2_{\beta} M_n(\beta)-\nabla^2_{\beta}M(\beta)|=o_P(1).$$
\end{lemma}

\begin{proof}
We have
\begin{eqnarray*}
\frac{n_i+1}{n}\nabla^{2}_{\beta}M_n(\beta) &=& \frac{1}{n}\sum_{i=1}^n \nabla_{\theta}(\ln c_{\theta})_{|\theta= \hat\psi_i} (\hat\UU_i)
\nabla^2_{\beta} \hat{\psi}(\beta,\beta'\ZZ_i)\,\hat\omega_{i,n} \\
&&+ \frac{1}{n}\sum_{i=1}^n \nabla^2_{\theta}(\ln c_{\theta})_{|\theta=\hat{\psi}_i}(\hat \UU_i) \nabla_{\beta} \hat{\psi}_i\nabla_{\beta'} \hat{\psi}_i \,\hat\omega_{i,n} \\
&=: & B_{1,n}(\beta)+B_{2,n}(\beta).
\end{eqnarray*}
\bqa
\lefteqn{ B_{1,n}(\beta)-\frac{1}{n}\sum_{i=1}^n \frac{\nabla_{\theta}c_{\theta}}{c_{\theta}}_{|\theta= \psi_i} (\UU_i)
\nabla^2_{\beta} \hat{\psi}(\beta,\beta'\ZZ_i)\hat\omega_{i,n}
 =  \frac{1}{n}\sum_{i=1}^n \left[ \nabla_{\theta}(\ln c_{\theta})_{|\theta= \hat\psi_i} (\hat\UU_i) \right. }\\
&-& \left. \nabla_{\theta}(\ln c_{\theta})_{|\theta= \hat\psi_i} (\UU_i)
+\nabla_{\theta}(\ln c_{\theta})_{|\theta= \hat\psi_i} (\UU_i) -
\nabla_{\theta}(\ln c_{\theta})_{|\theta= \psi_i} (\UU_i) \right]
\nabla^2_{\beta} \hat{\psi}(\beta,\beta'\ZZ_i)\,\hat\omega_{i,n} \\
&=&
 \frac{1}{n}\sum_{i=1}^n \left[ \nabla^2_{\uu,\theta}(\ln c_{\theta})_{|\theta= \hat\psi_i} (\UU^*_i).(\hat\UU_i - \UU_i)
+  \nabla^2_{\theta,\theta}(\ln c_{\theta})_{|\theta= \psi^*_i} (\UU_i).(\hat\psi_i - \psi) \right]
\nabla^2_{\beta} \hat{\psi}(\beta,\beta'\ZZ_i)\,\hat\omega_{i,n},
\eqa
for some $\UU^*_i$ and $\psi_i^*$ s.t.
$$ |\UU_i^* - \UU_i |<  |\hat\UU_i - \UU_i |,\;\; |\psi_i^* - \psi_i | < |\hat\psi - \psi_i|.$$
From Assumption~\ref{uchapeau} and with the same arguments as in the proof of Theorem~\ref{Th_consistency} (see the term $T_2(\beta)$), we get
$$  \sup_\beta | \frac{1}{n}\sum_{i=1}^n   \nabla^2_{\uu,\theta}(\ln c_{\theta})_{|\theta= \hat\psi_i} (\UU^*_i).(\hat\UU_i - \UU_i)
\nabla^2_{\beta} \hat{\psi}(\beta,\beta'\ZZ_i)\,\hat\omega_{i,n}| = o_P(1).$$
From Assumptions~\ref{hessc} and the uniform consistency of $\hat{\psi}(\beta,\beta'\zb)$ (see~(\ref{unifcons3})), we have
$$ \sup_\beta |
 \frac{1}{n}\sum_{i=1}^n \nabla^2_{\theta,\theta}(\ln c_{\theta})_{|\theta= \psi^*_i} (\UU_i).(\hat\psi_i - \psi)
\nabla^2_{\beta} \hat{\psi}(\beta,\beta'\ZZ_i)\,\hat\omega_{i,n}|=o_P(1),$$
and we deduce
$$ \sup_{\beta\in \Bc} |  B_{1,n}(\beta)-\frac{1}{n}\sum_{i=1}^n \frac{\nabla_{\theta}c_{\theta}}{c_{\theta}}_{|\theta= \psi_i} (\UU_i)
\nabla^2_{\beta} \hat{\psi}(\beta,\beta'\ZZ_i)\hat\omega_{i,n}|=o_P(1),$$
Invoking Assumption \ref{hess}, equation (\ref{unifcons3}), we get
$$\sup_{\beta\in \Bc} |  B_{1,n}(\beta)-\frac{1}{n}\sum_{i=1}^n \frac{\nabla_{\theta}c_{\theta}}{c_{\theta}}_{|\theta= \psi_i} (\UU_i)
\nabla^2_{\beta}\psi(\beta,\beta'\ZZ_i)\hat\omega_{i,n} |=o_P(1).$$
Since the score function is uniformly integrable (Assumption~\ref{A_moments}) and applying Lemma~\ref{tricky2} (or the dominated convergence theorem simply),
we can replace $\hat\omega_{i,n}$ by $\omega_i$.
Therefore, $\sup_\beta | B_{1,n}(\beta) - B_1(\beta)|= o_P(1)$, with
$$B_{1}(\beta)=\frac{1}{n}\sum_{i=1}^n \frac{\nabla_{\theta}c_{\theta}}{c_{\theta}}_{|\theta= \psi_i} (\UU_i)
\nabla^2_{\beta}\psi(\beta,\beta'\ZZ_i)\omega_{i}.$$

\mds

Similarly, one can deduce from Assumptions~\ref{uchapeau} and~\ref{hess} that $\sup_\beta |B_{2,n}(\beta) - B_2(\beta)|
 = o_P(1)$, with
$$B_2(\beta)=\frac{1}{n}\sum_{i=1}^n \nabla^2_{\theta,\theta}(\ln c_{\theta})_{|\theta=\psi_i}( \UU_i) \nabla_{\beta} \psi(\beta,\beta'\ZZ_i)\nabla_{\beta'} \psi(\beta,\beta'\ZZ_i)\omega_i.$$
From Assumption \ref{lipschitz} and (\ref{lip1}) and (\ref{lip2}) in Assumption \ref{hessc}, we can apply Example 19.7  and Theorem 19.4 in van der Vaart (2007) to deduce that
$$\sup_{\beta \in \Bc}|B_1(\beta)-E[B_1(\beta)]+B_2(\beta)-E[B_2(\beta)]|=o_P(1).$$
Since $(n_i+1)/n$ tends to $\Pb(\ZZ\in \Zc)$ a.e. and $\Sigma = \{E[B_1(\beta)] + E[B_2(\beta)]\}/\Pb(\ZZ\in \Zc)$, we obtain the result.
\end{proof}

\mds


\begin{lemma}
\label{tricky1}
Let $c_0(\uu,v)$ denote the first order partial derivative of $C^M_{\beta_0}(\uu|w)$ with respect to $w$ evaluated at point $w=v,$ where $C^{M}_{\beta}(\uu|w)$ denotes the conditional copula function of $\UU$ conditionally to $\beta'\ZZ$ and $\|\ZZ\|_{\infty}\leq M$ (that is $\ZZ\in \Zc$). We have
$$\nabla_{\beta}C_{\beta}(\uu|\beta'\ZZ)_{|\beta=\beta_0}=c_0(\uu,\beta_0'\ZZ)\left(\ZZ-E\left[\ZZ|\beta_0'\ZZ,\ZZ\in \Zc\right]\right).$$
\end{lemma}

\begin{proof}
The proof is similar to the proof of Lemma 5A in Dominitz and Sherman (2005), and of Lemma 3.4 in Lopez, Patilea and Van Keilegom (2013).
    Observe that
\begin{eqnarray*}
C^M_{\beta}(\uu|\beta'\ZZ)&=&E\left[\mathbf{1}_{\UU\leq \uu}|\beta'\ZZ,\ZZ\in \Zc\right] \\
&=& E\left[E\left[\mathbf{1}_{\UU\leq \uu}|\ZZ\right]|\beta'\ZZ,\ZZ\in \Zc\right] \\
&=& E\left[C^M_{\beta_0}(\uu|\beta_0'\ZZ)|\beta'\ZZ,\ZZ\in \Zc\right],
\end{eqnarray*}
where we used the single-index assumption for going from line 2 to line 3. Next, let
$$\Gamma_{\uu,\ZZ}(\beta_1,\beta_2)=E\left[C^M_{\beta_0}(\uu|\alpha(\ZZ,\beta_1)+\beta_2'\ZZ)|\beta_2'\ZZ,\ZZ\in \Zc\right],$$
where $\alpha(\ZZ,\beta_1)=\beta_0'\ZZ-\beta_1'\ZZ.$ Hence,
$C^M_{\beta}(\uu|\beta'\ZZ)=\Gamma_{\uu,\ZZ}(\beta,\beta).$ As a consequence,
$$\nabla_{\beta} C^M_{\beta}(\uu|\beta'\ZZ)_{|\beta=\beta_0}=\nabla_1 \Gamma_{\uu,\ZZ}(\beta,\beta_0)_{|\beta=\beta_0}+\nabla_2 \Gamma_{\uu,\ZZ}(\beta_0,\beta)_{|\beta=\beta_0},$$
where $\nabla_j$ represents the gradient vector with respect to $\beta_j.$ Observe that
$$\nabla_1 \Gamma_{\uu,\ZZ}(\beta,\beta_0)_{|\beta=\beta_0}=-E\left[\ZZ c_0(\uu,\beta_0'\ZZ)|\beta_0'\ZZ\right].$$
Moreover,
$\Gamma_{\uu,\ZZ}(\beta_0,\beta)=C^M_{\beta_0}(\uu|\beta'\ZZ),$ which leads to
$$\nabla_2 \Gamma_{\uu,\ZZ}(\beta_0,\beta)_{|\beta=\beta_0}=\ZZ c_0(\uu,\beta_0'\ZZ),$$
and the result follows.
\end{proof}

\begin{lemma}
\label{tricky}
Assume that the transformation $\Psi$ is Hadamard differentiable. Then, for all $v,$
$$\int \nabla_{\beta}\psi(\beta,\beta'\zb)_{\beta=\beta_0}d\mathbb{P}_{(\ZZ|\beta_0'\ZZ)}(\zb|v)=0.$$
\end{lemma}

\begin{proof}
Let $\dot\Psi(C(\cdot))[D(\cdot)]$ denote the Hadamard derivative of $\Psi$ at point $C,$ applied to function $D.$
Recall that
$$\psi(\beta,\beta'\zb)=\Psi(C^M_{\beta}(\cdot|\beta'\zb)).$$ Hence, using Lemma \ref{tricky1},
$$\nabla_{\beta}\psi(\beta,\beta'\zb)_{|\beta=\beta_0}=\left[\zb-E\left[\ZZ|\beta_0'\ZZ=\beta_0'\zb\right]\right]\dot\Psi\left(C^M_{\beta_0}(\cdot|\beta_0'\zb)\right)\left[c_0(\cdot|\beta_0'\zb)\right].$$ This shows that
$$\nabla_{\beta}\psi(\beta,\beta'\zb)_{|\beta=\beta_0}=\left[\zb-E\left[\ZZ|\beta_0'\ZZ=\beta_0'\zb\right]\right]\Lambda(\beta_0'\zb),$$ and the result of Lemma \ref{tricky} follows.
\end{proof}


Finally, Lemma~\ref{e_mason} invokes two propositions from Einmahl and Mason (2005), that we recall here.

\begin{proposition}
\label{prop_talagrand}
Let $\mathcal{G}$ denote a class of functions bounded by $1$, and let $\sigma^2_{\mathcal{G}}=\sup_{g\in \mathcal{G}}Var (g(\XX,\ZZ)).$ Then, for all $t>0,$
\begin{eqnarray*}
\mathbb{P}\left(\sup_{g\in \mathcal{G}}\left|\sum_{i=1}^n g(\XX_i,\ZZ_i)-E[g(\XX_i,\ZZ_i)]\right|\geq A_1( G_{\varepsilon}+t)\right)\leq 2\left\{\exp\left(-\frac{A_2 t^2}{n\sigma^{2}_{\mathcal{G}}}\right)+\exp(-A_2t)\right\},
\end{eqnarray*} for some universal constants $A_1$ and $A_2,$
and
$$G_{\varepsilon}:=E\left[\sup_{g\in \mathcal{G}}|\sum_{i=1}^n g(\XX_i,\ZZ_i)\varepsilon_i|\right],$$
where $(\varepsilon_i)_{1\leq i \leq n}$ are i.i.d. Rademacher variables independent from $(\XX_i,\ZZ_i)_{1\leq i \leq n}.$
\end{proposition}

\begin{proposition}
\label{prop_rade}
Assume that $\mathcal{G}$ is a class of functions satisfying the assumptions of Proposition \ref{prop_talagrand} and such that $N(\varepsilon,\mathcal{G})\leq C\varepsilon^{-\nu}$ for $C>0$ and $\nu>0.$ Moreover, assume that there exists
$\sigma^2\leq 1$ such that $\sup_{g\in \mathcal{G}}E[g(\XX,\ZZ)^2]\leq \sigma^2.$ Then,
$$G_{\varepsilon}\leq A n^{1/2}\sigma \log (1/\sigma).$$
\end{proposition}

\begin{proposition}
\label{born_exp}
For any $k=1,\cdots,p$, let $\hat{F}_k(x|\zz)$ denote the kernel estimator of the conditional distribution function $F_k(x|\zz)$ as given in Equation~(\ref{KernelFk}), i.e.
$ \hat{F}_k(x|\zz) = \hat N_k(x|\zz) / \hat{f}(\zz)$ with
$$\hat N_k(x|\zz) := \frac{1}{n}\sum_{i=1}^n \mathbf{1}_{X_k\leq x} \KK\left(\ZZ_i-\zb, \hh\right),\;\;\hat{f}(\zz) := \frac{1}{n}\sum_{i=1}^n \KK\left(\ZZ_i-\zb, \hh\right),$$
$$\KK\left(\ZZ_i-\zb, \hh\right) := \frac{1}{h_1\cdots h_p}\prod_{k=1}^p K_k\left( \frac{Z_{i,k}-z_k}{h_k}\right).$$

Define
$$P(t)=\mathbb{P}\left(\sup_{x\in \Rb,\zz\in \mathcal{Z}}\left|\hat{F}_k(x|\zz)-F_k(x|\zz)\right|\geq t\right).$$
and assume that
$$\sup_{x\in\Rb,\zz\in \mathcal{Z}}\left|\frac{E\left[\hat N_k(x|\zz)\right]}{E\left[\hat{f}(\zz)\right]}-F_k(x|\zz)\right|=b_n,$$
for some sequence $b_n\rightarrow 0.$
Then, for $B$ large enough and $t\geq \max(2b_n, Bn^{-1/2}\log (1/\min_k h_k)(h_1\cdots h_p)^{-1/2} )$, we have
$$P(t)\leq 4\left\{\exp\left(-\alpha nh_1\cdots h_p t^2\right)+\exp\left(-\beta nh_1\cdots h_p t\right)+\exp\left(-\gamma nh_1\cdots h_p \right)
+\exp\left(-\delta nh_1\cdots h_p\right)\right\},$$
for some positive constants $(\alpha,\beta,\gamma,\delta).$
\end{proposition}

\begin{proof}
Let
\begin{eqnarray*}
\tilde{P}(t) &=& \mathbb{P}\left(\sup_{x,\zz}\left|\hat{F}_k(x|\zz)-\frac{E\left[\hat N_k(x|\zz)\right]}{E\left[\hat{f}(\zz)\right]}\right|\geq t \right), \\
P_{1}(t) &=& \mathbb{P}\left(\sup_{x,\zz}\left|\hat N_k(x|\zz)-E\left[\hat N_k(x|\zz)\right]\right|\geq t\right), \\
P_{2}(t) &=& \mathbb{P}\left(\sup_{\zz}\left|\hat{f}(\zz)-E\left[\hat{f}(\zz)\right]\right|\geq t\right).
\end{eqnarray*}
We have
$$P(t)\leq \tilde{P}(t/2)+\mathbb{P}\left(\sup_{x,\zz}\left|\frac{E\left[\hat N_k(x|\zz)\right]}{E\left[\hat{f}(\zz)\right]}-F_k(\xx|\zz)\right|\geq t/2\right),$$ where the last probability is zero for $t/2\geq b_n.$ Hence, an upper bound for $P(t)$ can be deduced from an upper bound on $\tilde{P}(t).$

Define the classes of functions
\begin{eqnarray*}
\mathcal{G}_1=\left\{(x,\zz)\in \Rb\times \Zc\mapsto \mathbf{1}_{x'\leq x}\KK\left(\zz-\zz',\hh\right):\zz'\in \mathcal{Z}, \hh=(h_1,\ldots,h_p)\in \Rb^p_+,x'\in \Rb\right\},\,\text{and} \\
\mathcal{G}_2=\left\{\zz \in \Zc\mapsto \KK\left(\zz-\zz' ,\hh\right):\zz'\in \mathcal{Z}, \hh=(h_1,\ldots,h_p)\in \Rb^p_+, x\in \Rb\right\}.
\end{eqnarray*}
These two classes of functions satisfy the Assumption of Proposition \ref{prop_rade} with $\sigma^2 \propto \prod_{k=1}^p h_k$. Hence, we get, from Proposition~\ref{prop_talagrand},
$$P_j(t)\leq 2\left\{\exp\left(-C n(h_1\cdots h_p) t^2\right)+\exp\left(-Cn(h_1\cdots h_p)t\right)\right\},\;j=1,2,$$
for $t\geq A n^{-1/2} \log (1/\min_k h_k)/(h_1\cdots h_p)^{1/2},$ with $A$ large enough and $C>0$ some constant.

Decompose
\begin{eqnarray}
\tilde{P}(t) &\leq & \mathbb{P}\left(\sup_{x,\zz}\left|\frac{\hat N_k(x|\zz)-E\left[\hat N_k(x|\zz)\right]}{E\left[\hat{f}(\zz)\right]}\right|\geq t/2\right)+\mathbb{P}\left(\sup_{x,\zz}\left|\frac{\hat N_k(x|\zz)\left\{\hat{f}(\zz)-E\left[\hat{f}(\zz)\right]\right\}}{\hat{f}(\zz)E\left[\hat{f}(\zz)\right]}\right|\geq t/2\right) \nonumber \\
&\leq &\tilde{P}_1(t/2)+\tilde{P}_2(t/2). \label{p1+p2}
\end{eqnarray}
Next, recall that $\inf_{\zz\in\Zc}f(\zz)\geq f_0>0.$ Moreover, the bias of the kernel estimator of the density tends to $0$ uniformly on $\mathcal{Z}$,  i.e. $\sup_{\zz\in \Zc}\left|E\left[\hat{f}(\zz)\right]-f(\zz)\right|\rightarrow 0$ as $n$ grows to infinity. This, combined with the bound obtained on $P_2(t),$ shows that $\sup_{\zz\in \Zc}\left|\hat f(\zz)-f(\zz)\right|\rightarrow 0.$ Hence,
$\inf_{\zz\in \mathcal{Z}}E\left[\hat{f}(\zz)\right]>f_0/2$ for $n$ large enough, which leads to
\begin{equation}\label{ptilde1}
\tilde{P}_1(t/2)\leq P_1(tf_0/4),\end{equation} and
$$\tilde{P}_2(t/2)\leq \mathbb{P}\left(\sup_{x,\zz}\left|\hat{f}(\zz)-E\left[\hat{f}(\zz)\right]\right|\geq f_0/4\right)+\mathbb{P}\left(\sup_{x,\zz}\left|\hat N_k(x|\zz)\left\{\hat{f}(\zz)-E\left[\hat{f}(\zz)\right]\right\}\right|\geq f_0^2t/16\right).$$
Since $\sup_{x,\zz}|E\left[\hat N_k(x|\zz)\right]| \leq \|\KK\|_{\infty}<\infty$,
\begin{equation}
\tilde{P}_2(t/2) \leq  P_2(f_0/4)+P_2(f_0^2t/(16 \|\KK\|_{\infty})) . \label{ptilde2}
\end{equation}
Gathering (\ref{p1+p2}), (\ref{ptilde1}) and (\ref{ptilde2}) leads to the result of the proposition.
\end{proof}


\begin{thebibliography}{99}
\bibitem{Aas}
Aas, K., Czado, C., Frigessi, A. \& Bakken, H. (2009).
Pair-copula constructions of multiple dependence. {\it Insurance Math. Econom.} {\bf 44}, 182-198.
\bibitem{Abegaz}
Abegaz, F., Gijbels, I. \& Veraverbeke, N. (2012).
Semiparametric estimation of conditional copulas. {\it J. Multivariate Anal.} {\bf 110}, 43-73.
\bibitem{Acar2011}
Acar, E.F., Craiu, R.V. \& Yao, F. (2011).
Dependence Calibration in Conditional Copulas: A Nonparametric Approach. {\it Biometrics} {\bf 67}, 445-453.
\bibitem{Acar2013}
Acar, E.F., Craiu, R.V. \& Yao, F. (2013). Statistical testing of covariate effects in conditional copula models. {\it Electron. J. Stat.} {\bf 7}, 2822-2850.
\bibitem{AcarGenest}
Acar, E.F., Genest, C. \& J. Ne\v{s}lehov\'{a}, J. (2012).
Beyond simplified pair-copula constructions.
{\em J. Multivariate Anal.} {\rm 10}, $74-90$.
\bibitem{craiuSabeti}
Craiu, R. \& Sabeti, A. (2012).
In mixed company: Bayesian inference for bivariate conditional copula models with
discrete and continuous outcomes. {\it J. Multivariate Anal.} {\bf 110}, 106-120.
\bibitem{delecroix}
Delecroix, M. \& Hristache, M. (1999).
M-estimateurs semi-param\'etriques dans les mod\`eles \`a direction r\'ev\'elatrice
unique. {\it Bull. Belg. Math. Soc.} {\bf 6}, 161-185.
\bibitem{dominici}
Dominici, D. (2003).
The inverse of the cumulative standard normal probability function.
{\it Integral Transforms \& Spec. Funct.} {\bf 14}, 281-292.
\bibitem{duakritas} Du, Y. \& Akritas, M.G. (2002).  I.i.d. representations of the conditional Kaplan-Meier process for arbitrary distributions. {\it Math. Method. Statist.} \textbf{11}, 152-182.
\bibitem{derum}
Derumigny, A. \& Fermanian, J.-D. (2016).
About tests of the ``simplifying'' assumption for conditional copulas.
arXiv:1612.07349.
\bibitem{em2000}
Einmahl, U. \& Mason, D. (2000). An Empirical Process Approach to the Uniform Consistency of Kernel-Type Function Estimators.
 {\it J. Theoret. Probab.} {\bf{13}}, 1-37.
\bibitem{einmahl}
Einmahl, U. \& Mason, D. (2005). Uniform in bandwidth consistency of kernel-type function estimators. {\it Ann. Statist.} {\bf 33}, 1380-1403.
\bibitem{JDFMW}
Fermanian, J.-D. \& Wegkamp, M. (2012). Time-dependent copulae, {\it J. Multivariate Anal.} {\bf 110}, 19-29.
\bibitem{GNBG}
Genest, C., Ne\v{s}lehov\'{a}, J. \& Ben Ghorbal, N. (2011). Estimators based on Kendall's tau in multivariate copula models. {\it Aust. N.Z. J. Stat.} {\bf 53}, 157-177.
\bibitem{GVO}
Gijbels, I., Veraverbeke, N. \& Omelka, M. (2011). Conditional copulae, association measures and their applications.
{\it Comput. Statist. Data Anal.} {\bf 55}, 1919-1932.
\bibitem{hardleStoker}
H\"ardle, W. \& Stoker, T.M. (1989). Investigating smooth multiple regression by the method of average
derivatives. {\it J. Amer. Statist. Assoc.} {\bf 84}, 986-995.
\bibitem{HHI}
H\"{a}rdle, W., Hall, P. \& Ichimura, H. (1993). Optimal smoothing in single-index models. {\it Ann. Statist.}
{\bf 21}, 157-178.
\bibitem{Ichimura}
Ichimura, H. (1993). Semiparametric least squares (SLS) and weighted SLS estimation of single- index models. {\it J. Econometrics} {\bf 58}, 71-120.
\bibitem{joe1990}
Joe, H. (1990). Multivariate concordance. {\it J. Multivariate Anal.} {\bf 35}, 12-30.
\bibitem{joe}
Joe, H. (1997). {\it Multivariate Models and Dependence Concepts}, Volume
{\bf 73}. Chapman \& Hall, London.
\bibitem{jondeau}
Jondeau, E. \& Rockinger, M. (2006). The copula-garch model of conditional
dependencies: An international stock market application. {\it J. of Internat. Money and Finance} {\bf 25}, 827-853.
\bibitem{kendall}
Kendall, M.G. \& Babington Smith, B. (1940). On the method of paired comparisons. {\it Biometrika} {\bf 31}, 324-345.
\bibitem{klein}
Klein, R. L. \& Spady, R. H. (1993).
An efficient semiparametric estimator for binary response models. {\it Econometrica} {\bf 61}, 387-421.
\bibitem{lpvk}
Lopez, O., Patilea, V. \& van Keilegom, I. (2013). Single-index regression models in the presence of censoring depending on the covariates. {\it Bernoulli} {\bf 19}, 721-747.
\bibitem{nelsen}
Nelsen, R. (1998). {\it An introduction to copulas}. Lecture Notes in Statistics {\bf 139}. Springer, New-York.
\bibitem{Newey}
Newey, W. \& McFadden, D. (1994). Large sample estimation and
hypothesis testing, in {\it Handbook of Econometrics}, vol. IV,
2111-2245. Elsevier.
\bibitem{np} Nolan, D. \& Pollard, D. (1987). U-processes: Rates of convergence.
{\it Ann. Statist.} {\bf 15}, 780-799.
\bibitem{Patton2006}
Patton, A.J. (2006).
Modelling asymmetric exchange rate dependence. {\it Int. Econ. Rev.} {\bf 47}, 527-556.
\bibitem{Patton2009}
Patton, A.J. (2009).
Copula-based models for financial time series. In: {\it Handbook of Financial Time Series}, Springer, Berlin, 767-785.
\bibitem{Patton2012}
Patton, A.J. (2012). A review of copula models for economic time series. {\it J. Multivariate Anal.}, {\bf 110}, 4-18.
\bibitem{PSS}
Powell, J.L., Stock, J.H. \& Stoker, T.M. (1989). Semiparametric estimation of index coefficients.
{\it Econometrica} {\bf 57}, 1403-1430.
\bibitem{rodriguez}
Rodriguez, J.C. (2007). Measuring financial contagion: A copula approach.
{\it J. of Empirical Finance}, {\bf 14(3)} 401-423.
\bibitem{Sabeti}
Sabeti, A., Wei, M. \& Craiu, R. (2014).
Additive models for conditional copulas. {\it Stat} {\bf 3}, 300-312.
\bibitem{SS}
Schmid, F. \& Schmidt, R. (2007). Multivariate extensions of Spearman
rho and related statistics, {\it Statist. \& Probab. Lett.}, {\bf 77}(4), 407-416.
\bibitem{scott}
Scott, D.W. (2015). {\it Multivariate density estimation: theory, practice and visualization}. Wiley.
\bibitem{stoker}
Stoker, T. M. (1986). Consistent estimation of scaled coefficients. {\it Econometrica} {\bf 54}, 1461-1481.
\bibitem{tsu}
Tsukahara, H. (2005). Semiparametric estimation
in copula models. {\it Canad. J. Statist.} {\bf 33}, 357-375.
\bibitem{vdv}
van der Vaart, A. (2007). {\it Asymptotic Statistics}. Cambrige University Press.
\bibitem{vdvw}
van der Vaart, A. \& Wellner, J. (1996). {\it Weak convergence and empirical processes}. Springer, New-York.
\bibitem{verav}
Veraverbeke, N., Omelka, M. \& Gijbels, I. (2011). Estimation of a conditional copula and association measures.
{\it Scand. J. Statist.} {\bf 38}, 766-780.
\end{thebibliography}
\end{document}